\documentclass{amsart}
\usepackage{graphicx} 
\usepackage{url, amsthm, amsmath}
\usepackage{amssymb}
\usepackage{appendix}

\newcommand{\N}{\mathbb{N}}
\newcommand{\norm}[1]{\| #1 \|}
\newcommand{\C}{\mathbb{C}}

\newcommand{\unit}{\mathbf{1}}
\renewcommand{\star}{\ast}
\newcommand{\rem}{\mbox{rem}}
\newcommand{\boldA}{\mathbf{A}}
\newcommand{\boldB}{\mathbf{B}}
\newcommand{\mvn}{\sim_{\operatorname{mvn}}}

\newcommand{\zerom}{\mathbf{0}}

\newcommand{\diam}{\mbox{diam}}
\newcommand{\ran}{\mbox{ran}}

\newcommand{\primes}{\mathbb{P}}
\newcommand{\Q}{\mathbb{Q}}
\newcommand{\abs}[1]{\left\vert #1 \right\vert}
\newcommand{\embed}{\mathcal{E}}

\theoremstyle{plain}
\newtheorem{theorem}{Theorem}[section]
\newtheorem{lemma}[theorem]{Lemma}
\newtheorem{fact}[theorem]{Fact}

\newtheorem{corollary}[theorem]{Corollary}

\newtheorem{proposition}[theorem]{Proposition}

\newtheorem{question}[theorem]{Question}

\newtheorem{maintheorem}{Main Theorem}

\theoremstyle{definition}
\newtheorem{definition}[theorem]{Definition}
\newtheorem{remark}[theorem]{Remark}

\newcommand{\cstar}{$\mathrm{C}^*$}

\title{Computable $K$-theory for \cstar-algebras:  UHF algebras}
\author[C. J. Eagle]{Christopher J. Eagle${}^1$} 
\address[C. J. Eagle]
{Department of Mathematics and Statistics, University of Victoria. PO BOX 1700 STN CSC, Victoria, British Columbia, Canada. V8W 2Y2}%
\email{eaglec@uvic.ca}
\urladdr{http://www.math.uvic.ca/~eaglec}
\thanks{${}^1$ Supported by NSERC Discovery Grant RGPIN-2021-02459}

\author[I. Goldbring]{Isaac Goldbring${}^2$}
\address[I. Goldbring]{University of California, Irvine}
\email{goldbring@uci.edu}
\urladdr{https://www.math.uci.edu/~isaac}
\thanks{${}^2$  Supported by NSF grant DMS-2054477.}

\author[T. H. McNicholl]{Timothy H. McNicholl${}^3$}
\address[T. H. McNicholl]{Iowa State University}
\email{mcnichol@iastate.edu}
\urladdr{https://faculty.sites.iastate.edu/mcnichol/}
\thanks{${}^3$}

\author[R. Miller]{Russell Miller${}^4$}
\address[R. Miller]{City University of New York}
\email{russell.miller@qc.cuny.edu}
\urladdr{https://qcpages.qc.cuny.edu/~rmiller/}
\thanks{${}^4$}

\begin{document}

\begin{abstract}
We initiate the study of the effective content of $K$-theory for \cstar-algebras.  We prove that there are computable functors which associate, to a computably enumerable presentation of a \cstar-algebra $\boldA$, computably enumerable presentations of the abelian groups $K_0(\boldA)$ and $K_1(\boldA)$.  When $\boldA$ is stably finite, we show that the positive cone of $K_0(\boldA)$ is computably enumerable.  We strengthen the results in the case that $\boldA$ is a UHF algebra by showing that the aforementioned presentation of $K_0(\boldA)$ is actually computable.  In the UHF case, we also show that $\boldA$ has a computable presentation precisely when $K_0(\boldA)$ has a computable presentation, which in turn is equivalent to the supernatural number of $\boldA$ being lower semicomputable; we give an example that shows that this latter equivalence cannot be improved to requiring that the supernatural number of $\boldA$ is computable.  Finally, we prove that every UHF algebra is computably categorical. 
\end{abstract}

\maketitle

\section*{Introduction}

One of the most important invariants attached to a \cstar-algebra $\boldA$ is the $K$-theory of $\boldA$, which consists of a partially 
ordered abelian group $K_0(\boldA)$ and another abelian group $K_1(\boldA)$.  Initially arising in algebraic topology, $K$-theory has played a
central role in the classification program for \cstar-algebras.  With hindsight, Glimm's classification of Uniformly Hyperfinite (UHF) 
algebras \cite{Glimm.1960} in terms of supernatural numbers can be viewed as showing that if $\boldA$ and $\boldB$ are UHF algebras, then 
$\boldA \cong \boldB$ as \cstar-algebras if and only if $K_0(\boldA) \cong K_0(\boldB)$ as abelian groups.  Elliott then classified 
Approximately Finite-Dimensional (AF) algebras by their \emph{ordered} $K_0$ groups.  This result launched a line of research in classifying 
increasingly general classes of \cstar-algebras through $K$-theoretic data, culminating in the recent proof \cite{CastillejosEtAl.2021} that a
very wide class of \cstar-algebras (namely, the simple, separable, nuclear, unital, $\mathcal{Z}$-stable algebras satisfying the Universal Coefficient
Theorem) are classifiable by data built on their $K$-theory.
\footnote{It may seem puzzling that this large set of conditions yields a very large class of \cstar-algebras.  However, it is known that many different methods for constructing \cstar-algebras give rise to algebras that do satisfy the hypotheses of the classification theorem; see \cite{White.2023} for a survey of the \cstar-algebra classification program.}

In this paper, we address the question of whether the $K$-theory of a separable \cstar-algebra can be effectively computed.  Our main results along these lines involve the notions of presentations of \cstar-algebras and groups.  By a presentation $\boldA^\#$ of $\boldA$, we mean a countable  sequence from $\boldA$ that generates $\boldA$ as a \cstar-algebra (so only separable \cstar-algebras have presentations); the presentation is computable if there is an algorithm such that, upon input a $*$-polynomial $p(\vec x)$ with coefficients from $\mathbb{Q}(i)$ and no constant term, a tuple $\vec a$ from the aforementioned countable sequence, and $k \in \N$, returns $q\in \mathbb{Q}^{>0}$ such $|\|p(\vec a)\|-q|< 2^{-k}$.  If instead the algorithm only returns a computable sequence of upper bounds converging to $\|p(\vec a)\|$, then we say that the presentation is computably enumerable.  The definition of a presentation of a group is slightly more involved; see Subsection \ref{subsec:MvN} for the precise definition as well as what it means for a presentation of a group to be computable or computably enumerable.  (We note that the latter condition coincides with the notion of a recursive presentation of a group from algorithmic group theory.)     

Our main results are the following:

\begin{maintheorem}\label{mthm:K0}
There is a computable functor $K_0$ so that $K_0(\boldA^\#)$ is a computably enumerable presentation of the abelian group $K_0(\boldA)$ whenever $\boldA^\#$ is a computably enumerable presentation of a unital separable \cstar-algebra $\boldA$.  Moreover, if $\boldA$ is stably finite, then the positive cone of $K_0(\boldA)$ is a computably enumerable subset of $K_0(\boldA^\#)$.
\end{maintheorem}

\begin{maintheorem}\label{mthm:K1}
There is a computable functor $K_1$ so that $K_1(\boldA^\#)$ is a computably enumerable presentation of the abelian group $K_1(\boldA)$ 
whenever $\boldA^\#$ is a computably enumerable presentation of a unital separable \cstar-algebra $\boldA$.
\end{maintheorem}

Presently, we do not know if $K_0(\boldA^\#)$ is computable provided $\boldA^\#$ is computable. In the case of UHF algebras (for which computably enumerable presentations are automatically computable), we do indeed obtain stronger results.

\begin{maintheorem}
If $\boldA$ is a UHF algebra with a computable presentation $\boldA^\#$, then $K_0(\boldA^\#)$ is a computable presentation of $K_0(\boldA)$ and the positive cone $K_0(\boldA)^+$ is a computable subset of $K_0(\boldA^\#)$.
\end{maintheorem}

UHF algebras can be classified either by their $K_0$ groups or by their \emph{supernatural number} (see Section \ref{sec:UHF} below for the definition).

\begin{maintheorem}
Let $\boldA$ be a UHF algebra.  The following are equivalent:
\begin{enumerate}
\item{The \cstar-algebra $\boldA$ has a computable presentation.}
\item{The abelian group $K_0(\boldA)$ has a computable presentation.}
\item{The abelian group $K_0(\boldA)$ has a computably enumerable presentation.}
\item{The supernatural number of $\boldA$ is lower semicomputable.}
\end{enumerate}
\end{maintheorem}

Here, the supernatural number is lower semicomputable if, roughly speaking, it can be computably approximated from below; see Definition \ref{defn:frombelow} for a more precise statement.  One might have expected that a UHF algebra is computably presentable if and only if its supernatural number is in fact computable.  However, we give an example of a lower semicomputable supernatural number with Turing degree $\mathbf{0}''$ (and so, in particular, is not computable); by our theorem above, the associated UHF algebra is computably presentable. 
 
 Finally, we prove the following:

\begin{maintheorem}
Every UHF algebra is computably categorical.  That is, any two computable presentations of a given UHF algebra are computably isomorphic.
\end{maintheorem}

We develop the effective $K$-theory for the larger class of AF algebras in a sequel currently in preparation.

The paper is organized as follows.  In Section \ref{sec:prelim} we provide the necessary background in computable structure theory for \cstar-algebras.  In Section \ref{sec:K0}, we construct the functor $K_0$, while in Section \ref{sec:K1} we construct $K_1$.  Finally, in Section \ref{sec:UHF}, we establish results specific to UHF algebras.  The paper concludes with an appendix briefly discussing how our results extend to more general categories of presentations of \cstar-algebras.

We assume that the reader is familiar with the basics of the theory of \cstar-algebras and of classical computability theory. 
In the next section, we will present only the preliminaries from computable structure theory and computable analysis that 
we will directly use.  For a more thorough coverage of computability theory, we recommend \cite{Cooper.2004}.  For an 
expansive treatment of computable analysis, we recommend \cite{brattka.hertling.2021}.  
Finally, for additional background on \cstar-algebras and $K$-theory, see \cite{Rordam.Larsen.Laustsen.2000}.

Throughout the paper, all \cstar-algebras are assumed to be separable.

\section{Preliminaries}\label{sec:prelim}

\subsection{Presentations of \cstar-algebras}

\begin{definition}
Let $\boldA$ be a \cstar-algebra.  A \emph{presentation} of $\boldA$ is a sequence $(a_0, a_1, \ldots)$ of points in $\boldA$ such that the *-subalgebra of $\boldA$ generated by $\{a_i : i \in \mathbb{N}\}$ is dense in $\boldA$.  The points $a_0, a_1, \ldots$ are called the \emph{special points} of the presentation, while the points of the form $p(a_0,\ldots,a_n)$, where $p(x_0,\ldots,x_n)$ is a $*$-polynomial with coefficients from $\mathbb{Q}(i)$ with no constant term (as $n$ varies), are called the \emph{rational points} of the presentation.
\end{definition}

We note that this definition of ``presentation" is not the same as a presentation by generators and relations, but there are connections between the two notions (see \cite[Subsection 2.2]{Goldbring.2024+}).
\it Throughout this paper, unless explicitly stated otherwise, $\boldA$ denotes a unital \cstar-algebra (although at some points we will consider 
nonunital algebras).\rm\  For a \cstar-algebra $\boldA$, $\boldA^\#$ denotes a presentation of $\boldA$.

Throughout, we always consider $\mathbb{C}$ as equipped with its \emph{standard presentation} whose special points are $\mathbb{Q}(i)$.

By standard techniques (see, for example, \cite[Section 2]{burton2024computable}), we can obtain an effective (but typically non-injective) list of the rational points of a presentation of a \cstar-algebra.  When we say, for example, that an algorithm takes a rational point $q$ of $\boldA^\#$ as input, we really mean that we have fixed an effective list $(q_i)_{i \in \mathbb{N}}$ of rational points of $\boldA^\#$ and the algorithm takes input an index $i$ for which $q = q_i$.

By the Church-Turing thesis, every algorithm can be represented as a Turing machine.  An \emph{index} (or \emph{code}) of an 
algorithm is a natural number that codes a representation of the algorithm as a Turing machine.

\begin{definition}\label{def:pres.comp.etc}
Fix a presentation $\boldA^\#$.
\begin{enumerate}

    \item $\boldA^\#$ is \emph{computable} if there is an algorithm that, given a rational point $q$ and a $k \in \N$, outputs a rational number $r$ such that $\abs{\norm{q} - r} < 2^{-k}$; an index of such an algorithm is an \emph{index} of $\boldA^\#$.
    
    \item $\boldA^\#$ is \emph{left c.e.} if there is an algorithm that, given a rational point $q$, enumerates an increasing sequence $(r_n)_{n \in \mathbb{N}}$ of rational numbers such that $\lim_n r_n = \norm{q}$; an index of such an algorithm is a 
    \emph{left c.e. index} of $\boldA^\#$.

    \item $\boldA^\#$ is \emph{right c.e.} if there is an algorithm that, given a rational point $q$, enumerates a decreasing sequence $(r_n)_{n \in \mathbb{N}}$ of rational numbers such that $\lim_{n\to\infty}r_n = \norm{q}$; an index of such an algorithm is a 
    \emph{right c.e. index} of $\boldA^\#$.
\end{enumerate}
\end{definition}

It is immediate from the definitions that a presentation is computable if and only if it is both left and right c.e.  Fox \cite[Theorem 3.14]{fox2022computable} has shown that if $\boldA$ is simple, then every right c.e. presentation of $\boldA$ is computable.
In the sequel, we shall refer to right-c.e. presentations simply as \emph{c.e. presentations}; for some justification behind this decision, see \cite[Theorem 3.3]{fox2022computable}.

With each presentation of $\boldA$, there is an associated class of computable points.  These are defined as follows.

\begin{definition}
An element $a\in \boldA$ is a \emph{computable point of $\boldA^\#$} if there is an algorithm such that, given $k \in \N$,
returns a rational point $b$ of $\boldA^\#$ such that $\|a-b\|< 2^{-k}$; an index of such an algorithm is a 
\emph{$\boldA^\#$-index for $a$}. 
\end{definition}

In other words, a computable point of $A^\#$ is one that can be effectively approximated by rational points of 
$\boldA^\#$ with arbitarily good precision.

Similarly, with each presentation of $\boldA^\#$, there is an associated class of computable sequences.

\begin{definition}\label{def:comp.seq}
A sequence $(p_n)_{n \in \N}$ of points in $\boldA$ is a \emph{computable sequence of $\boldA^\#$} if 
there is an algorithm that, given any $k,n \in \N$, computes a rational vector $\rho$ of $\boldA^\#$
so that $\norm{\rho - p_n} < 2^{-k}$.  An index of such an algorithm is a $\boldA^\#$-index of 
$(p_n)_{n \in \N}$.
\end{definition}

We regard the collection of presentations of \cstar algebras as a category in which the morphisms are the
computable $*$-homomorphisms.  These maps are defined as follows.

\begin{definition}
A bounded linear map $\phi:\boldA_0\to \boldA_1$ is a \emph{computable map from $\boldA_0^\#$ to $\boldA_1^\#$} if there is an algorithm that,
given a rational point $a$ of $\boldA_0^\#$ and a $k \in \N$, computes a rational point $b$ of $\boldA_1^\#$ such that 
$\|\phi(a)-b\|< 2^{-k}$; 
an index of such an algorithm is an \emph{$(A_0^\#, A_1^\#)$-index} of $\phi$.
\end{definition}

Put another way, $\phi$ is computable if the $\phi$-image of every rational point of $\boldA_0^\#$ is a computable point of $\boldA_1^\#$ 
whose index can be computed from the rational point of $\boldA_0^\#$.   
It is easily verified that the inverse of a computable $*$-isomorphism is computable as well.
We note that there is a general definition of a computable map between presentations of metric structures;
when the function is a bounded linear map between normed spaces, the general definition specializes to the definition given here.   

Suppose $\boldA$ and $\boldB$ are \cstar algebra.  Then $\boldA \times \boldB$ is a \cstar algebra, where the operations are defined
coordinate-wise and the norm of $(a,b)$ is $\max\{\norm{a}, \norm{b}\}$.  
Given presentations $\boldA^\#$ and $\boldB^\#$, we define $\boldA^\# \times \boldB^\#$ to be the presentation of 
$\boldA \times \boldB$ whose special points are the points formed by pairing rational points of $\boldA^\#$ and $\boldB^\#$.
Here, we view these points as enumerated by $\N$ by fixing some computable bijection between 
$\N$ and $\N^2$.  Note that  $\boldA^\# \times \boldB^\#$ is in fact the product in the category of presentations of \cstar algebras.  If $\boldA^\#$ and $\boldB^\#$ are computable (resp. c.e.), then it is readily verified that $\boldA^\# \times \boldB^\#$ is also computable (resp. c.e.).


\subsection{Matrix amplifications}

The definition of $K_0(\boldA)$ involves projections in matrix algebras over $\boldA$, so we must consider how to produce a (computable) presentation of $M_n(\boldA)$ from a (computable) presentation of $\boldA$.  We begin by fixing some notations.  Throughout, we fix a positive integer $n$.

First, we let $(E^n_{r,s}(\boldA))_{r,s \in \{1, \ldots, n\}}$ denote 
the standard system of matrix units for $M_n(\boldA)$, that is, $E^n_{r,s}(\boldA)$ is the element of $M_n(\boldA)$ whose $(r,s)$ entry is $1_\boldA$ and with all other entries $0$; when $\boldA=\mathbb{C}$, we simply write $E^n_{r,s}$ for the matrix units of $M_n(\C)$.

Throughout this paper, we let $\|\cdot\|$ denote the \cstar-norm on $M_n(\boldA)$.  Occasionally, we will have the need to consider other norms on $M_n(\boldA)$.  For $A = (a_{r,s})_{r,s \in \{1, \ldots, n\}} \in M_n(\boldA)$, we set $\norm{A}_{\max} = \max_{r,s} \norm{a_{r,s}}$ and $\norm{A}_1 = \sum_{r,s} \norm{a_{r,s}}$.  It is well-known that 
$\norm{A}_{\max} \leq \norm{A} \leq \norm{A}_1$ and thus 
$\frac{1}{n^2}\norm{A}_1 \leq \norm{A}$.




We let $()^3:\N\to \N^3$ denote a computable bijection with component functions $()^3_0$, $()^3_1$, $()^3_2$, so $(k)^3 = ((k)^3_0, (k)^3_1, (k)^3_2))$ for all $k\in \N$.

For $m$ a nonzero integer, let $\rem(m,n)$
denote the remainder when $m$ is divided by $n$.  For each $k \in \N$, set:
\begin{eqnarray*}
\phi_0(k) & = & (k)^3_0\\
\phi_1(n,k) & = & 1 + \rem((k)^3_1, n)\\
\phi_2(n,k) & = & 1 + \rem( (k)^3_2, n).
\end{eqnarray*}

We are finally ready to define the presentation on $M_n(\boldA)$ induced by a presentation on $\boldA$.  

\begin{definition}\label{def:pres.Mn}
Suppose that $\boldA^\#$ is a presentation of $\boldA$ and let $v_m$ denote the $m$-th special point of $\boldA^\#$.  We define $M_n(\boldA^\#)$ to be the presentation of 
$M_n(\boldA)$ whose $k$-th special point is $v_{\phi_0(k)} E^n_{\phi_1(n,k), \phi_2(n,k)}(\boldA)$.
\end{definition}

Since $\norm{A}_{\max} \leq \norm{A} \leq \norm{A}_1$, 
it is straightforward to verify that $M_n(\boldA^\#)$ is indeed a presentation of $M_n(\boldA)$.  However, we have a much stronger result.

\begin{proposition}\label{prop:MnA.pres.comp}
If $\boldA^\#$ is c.e. (resp. computable), then $M_n(\boldA^\#)$ is also c.e. (resp. computable).
\end{proposition}

\begin{proof}
Set $\boldA_1 = M_n(\C) \otimes \boldA$.  
Let $\boldA_1^\#$ be the presentation of $\boldA_1$ whose $\langle m, k \rangle$-th distinguished 
point is the tensor of the $m$-th distinguished point of the standard presentation of $M_n(\C)$ (defined as in Definition \ref{def:pres.Mn} using the standard presentation of $\mathbb{C}$)
and the $k$-th distinguished point of $\boldA^\#$, where $\langle\cdot,\cdot\rangle$ is the usual computable pairing of $\N^2$ with $\N$.  
By \cite[Lemma 2.4 and Proposition 2.6]{Goldbring.2024+}, $\boldA_1^\#$ is c.e. (resp. computable) if 
$\boldA^\#$ is c.e. (resp. computable).  

Let $\psi:M_n(\boldA)\to M_n(\C) \otimes \boldA$ be the usual isomorphism.  Note that $\psi$ computably maps each 
rational point of $M_n(\boldA^\#)$ to a rational point of 
$\boldA_1^\#$.  Thus, since $\psi$ is an isomorphism, $\psi$ is 
a computable map from $M_n(\boldA^\#)$ to $\boldA_1^\#$.   Therefore, 
$M_n(\boldA^\#)$ is c.e. (resp. computable) if $\boldA^\#$ is c.e. (resp. computable).
\end{proof}

For the remainder of the paper, whenever we have a presentation $\boldA^\#$ of a \cstar-algebra $\boldA$, we always consider $M_n(\boldA)$ equipped with the presentation $M_n(\boldA^\#)$ we have defined here.

\subsection{C.e. open and c.e. closed sets}

Certain subsets of $\boldA$ with particular computability properties will play significant roles for us.  First, by a \emph{rational open ball} of $\boldA^\#$ we mean an open ball in $\boldA$ with rational radius and centered at a rational point of $\boldA^\#$.

\begin{definition}\label{def:c.e.clsd}
If $C \subseteq \boldA$ is closed, we say that $C$ is a \emph{c.e.-closed subset of $\boldA^\#$} if 
the set of all open rational balls of $\boldA^\#$ that intersect $C$ is c.e.
\end{definition}

\begin{definition}\label{def:c.e.open}
If $U \subseteq \boldA$ is open, we say that $U$ is a \emph{c.e.-open set of $\boldA^\#$} if 
there is a c.e. set $\mathcal{S}$ of rational open balls so that 
$U = \bigcup \mathcal{S}$.
\end{definition}

The import of a set $U$ being c.e.-open is that, provided the presentation is c.e., if a computable point $a$ belongs to $U$, then an effective search will eventually lead us to this fact.  Indeed, one must simply search for a rational open ball $B(x;r)$ in the c.e. set $\mathcal{S}$ from the definition and a rational point $b$ such that $\|a-b\|<\epsilon$ for rational $\epsilon$ small enough so that $\|x-b\|<r-\epsilon$.

By an index of a c.e. open set we mean an index for the c.e. set of open balls whose union comprises the set; the index of a c.e. closed set is defined analogously.

The following fact is probably folklore; we include a proof for the sake of the reader.  In the following proof (as well as later in the 
paper), it will be useful to observe that if $\boldA^\#$ is a c.e. presentation, then the subset relation between rational open balls of 
$\boldA^\#$ is also c.e.  In fact, this holds whenever $\boldA$ is a vector space, because in that case $B(\rho_0, r_0) \subseteq B(\rho_1,
r_1)$ if and only if $\norm{\rho_0 - \rho_1} + r_0 < r_1$.

\begin{lemma}\label{intersection}
Suppose that $\boldA^\#$ is a c.e. presentation.  Further suppose that $U$ is a c.e. open subset of $\boldA^\#$ and $C$ is a c.e. closed 
subset of $\boldA^\#$ such that $U\cap C\not=\emptyset$.  Then $U\cap C$ contains a computable point.  Moreover, an index of this computable
point can be computed from indices of $U$ and $C$.
\end{lemma}

\begin{proof}
First search for the first rational ball $B_0$ of radius less than $1$ that is contained in $U$ and intersects $C$.  Then search for the first
rational ball $B_1$ of radius less than $1/2$ contained in $B_0$ that intersects $C$.  Continuing in this manner, the centers of these balls 
converge to a point of $C$ that is necessarily computable.
\end{proof}

In continuous model theory, the definable sets in metric structures are characterized by a condition known as \emph{weak stability} 
(see \cite[Definition 3.2.4]{FarahEtAll.2021}).  Many proofs that sets are c.e.-closed can be described in terms of a computable version of 
weak stability.  The relevant notion of computable weak stability was isolated by Fox, Goldbring, and Hart \cite[Definition 5]
{FoxGoldbringHart.2024+}.  We recall it here (in a slighlty different guise) and take the opportunity to show that computable weak stability 
always gives rise to c.e.-closed sets.

\begin{definition}\label{def:computably.weakly.stable}
Suppose that $\vec{x} = (x_1, \ldots, x_N)$ is a tuple of variables, $C_1, \ldots, C_N \in \N$, 
and $p_1(\vec{x}), \ldots, p_M(\vec{x})$ are rational $*$-polynomials.  Consider the set 
$$\mathcal{R} = \{p_i(\vec{x})=0 \ : \ i=1,\ldots,M\}\cup\{\norm{x_j} \leq C_j \ : \ j=1,\ldots,N\}$$ of relations.
\begin{enumerate}
    \item For a \cstar algebra $\boldB$ and $w_1, \ldots, w_N \in \boldB$, 
    let $\mathcal{R}^{\boldB}$ denote the quantity 
    $$
    \max(\{\norm{p_j(\vec{w})}\ :\ j \in \{1, \ldots, M\}\} \cup \{\norm{w_j}- C_j\ : j \in \{1, \ldots, N\}\}$$
    and write $\boldB \models \mathcal{R}(\vec{w})$ if 
    $p_j(\vec{w}) = 0$ for all $j \in \{1, \ldots, M\}$ and $\norm{w_j} \leq C_j$ for all $j \in \{1, \ldots, N\}$.

    \item A function $g : \N \rightarrow \N$ is a \emph{modulus of weak stability} for 
    $\mathcal{R}$ provided that, for every $k \in \N$, every \cstar algebra $\boldB$, 
    and every $w_1, \ldots, w_N \in \boldB$, if $\mathcal{R}^\boldB(\vec{w}) < 2^{-g(k)}$, 
    then there exists $z_1, \ldots, z_N \in \boldB$ so that 
    $\max_j \norm{w_j - z_j} < 2^{-k}$ and $\mathcal{B} \models \mathcal{R}(\vec{z})$.

    \item We say that $\mathcal{R}$ is \emph{weakly stable} if it has a modulus of weak stability and \emph{computably weakly stable} if it has a computable modulus of weak stability.
\end{enumerate}
\end{definition}


\begin{theorem}\label{thm:cws.ce.closed}
Suppose that $\boldA^\#$ is a c.e. presentation and suppose that $E \subseteq \boldA^N$ is defined by a computably weakly stable set of 
relations.  Then $E$ is a c.e.-closed subset of $(\boldA^\#)^N$.
\end{theorem}

\begin{proof}
Let $\mathcal{R} = \{p_i(\vec{x})=0 \ : \ i=1,\ldots,M\}\cup\{\norm{x_j} \leq C_j \ : \ j=1,\ldots,N\}$
be a computably weakly stable set of relations that defines $E$.   
Fix a computable modulus of weak stability $g$ for $\mathcal{R}$.  

Let $S_0$ be the set of all rational balls of $(\boldA^\#)^N$ of the form $B(\vec{a}; 2^{-k})$, 
where $\max_i \norm{p_i(\vec{a})} < 2^{-g(k)}$ and $\norm{a_j} < C_j + 2^{-g(k)}$ for all $j=1,\ldots,N$. 
It is clear that $S_0$ is a c.e. set.

We now note that, for every $\vec{u} \in E$ and every $\epsilon > 0$, there is a ball $B(\vec{a}; 2^{-k}) \in S_0$ 
such that $u \in B(\vec{a}; r)$ and $2^{-k} < \epsilon$.  To see this, choose $k \in \N$ so that 
$2^{-k} < \epsilon$. 
Using the continuity of the polynomials $p_i(\vec{x})$, we may choose a tuple of rational vectors $\vec{a}$ 
sufficiently close to $\vec{u}$ so that $\norm{\vec{u} - \vec{a}} < 2^{-k}$, $\max_i \norm{p_i(\vec{a})} < 2^{-g(k)}$, 
and $\max_j \norm{a_j} - C_j < 2^{-g(k)}$.  Then $B(\vec{a}; 2^{-k})$ is as required.

The theorem now follows by noticing that for a rational open ball $B$ of $(\boldA^N)^\#$, $B$ contains an element of $E$ if and only if $B$ 
includes an element of $S_0$.  Indeed, first suppose that $B$ contains the element $\vec{e}$ of $E$ and set 
$\epsilon = d(p, \boldA^N \setminus B)$.  By the previous paragraph, there is $B' \in S_0$ so that $\vec e \in B'$ and 
$\diam(B') < \epsilon$, whence $B' \subseteq B$.  In the other direction, suppose that $B$ includes 
the element $B(\vec a;r)$ of $S_0$.  Then by the definitions of $g$ and $S_0$, there is $\vec e\in E$ such that 
$\|\vec a-\vec e\|<r$.  Hence, $\vec e\in B(\vec a;r)\subseteq B$.   
\end{proof}

  Of particular importance for us are the relations defining projections and systems of matrix units, both of which are shown to be given by computably weakly stable relations in \cite[Examples 12]{FoxGoldbringHart.2024+}.  We therefore have the following two corollaries.

\begin{corollary}\label{prop:c.e.clsd}
If $\boldA^\#$ is a c.e. presentation of a \cstar-algebra $\boldA$, then the set of all projections in $\boldA$ is a c.e.-closed subset of $\boldA^\#$.
\end{corollary}

A matrix $(a_{r,s})_{r,s} \in M_n(\boldA)$ is an
\emph{$n \times n$ system of matrix units for $\boldA$} if 
$a_{r,s}^* = a_{s,r}$ and if $a_{r,s}a_{r',s'} = \delta_{s,r'} a_{r,s'}$. 
We may also regard an $n \times n$ system of 
matrix units as a point in $\boldA^{n^2}$.

\begin{corollary}\label{cor:matrix.units.c.e.clsd}
If $\boldA^\#$ is a c.e. presentation of a \cstar-algebra $\boldA$, then for every positive integer $n$,
the set of all $n \times n$ systems of
matrix units of $\boldA$ is a c.e.-closed subset of $(\boldA^\#)^{n^2}$ and of 
$M_n(\boldA^\#)$.
\end{corollary}

\begin{proof}
Suppose $\boldA^\#$ is a c.e. presentation of a \cstar-algebra $\boldA$, and 
let $E$ denote the set of $n \times n$ systems of matrix units for $\boldA$.  
That $E$ is a c.e. closed subset of $(\boldA^\#)^{n^2}$ follows directly now from 
the aforementioned result of \cite{FoxGoldbringHart.2024+}  while the fact that it is a c.e. closed subset of $M_n(\boldA^\#)$
 follows from the inequalities $\norm{A}_{\max} \leq \norm{A} \leq \norm{A}_1$.
\end{proof}

In Section \ref{sec:UHF}, we will need the following variant of the previous corollary:

\begin{corollary}\label{cor:matrix.units.c.e.clsd.variant}
If $\boldA^\#$ is a c.e. presentation of a \cstar-algebra $\boldA$, then for every positive integer $n$,
the set of all $n \times n$ systems $(a_{r,s})_{r,s}$ of
matrix units of $\boldA$ for which $\sum_{r=1}^n a_{rr}=1_\boldA$ is a c.e.-closed subset of $(\boldA^\#)^{n^2}$ and of 
$M_n(\boldA^\#)$.
\end{corollary}

\begin{proof}
Let $\mathcal{R}$ be the relations obtained from the above relations for matrix units with the additional relation $\sum_{r=1}^n x_{rr}=1$.  Fix $\epsilon>0$ and suppose that $\mathcal{R}^\boldA(a_{r,s})<\delta$ for $\delta>0$ to be specified shortly.  If $\delta$ is sufficiently small, by the previous corollary, there are actual matrix units $(b_{r,s})_{r,s} \in M_n(\boldA)$ such that $\|a_{r,s}-b_{r,s}\|<\epsilon$ for all $r,s=1,\ldots,n$.  Note that the $b_{rr}$'s are orthogonal projections, whence $1_\boldA-\sum_{r=1}^n b_{rr}$ is a projection.  Moreover, $\|1_\boldA-\sum_{r=1}^n b_{rr}\|< n\epsilon+\delta$.  Defining $b'_{rs} =b_{rs}$ if $(r,s)\not=(n,n)$ and $b_{nn}'=b_{nn}+(1_\boldA-\sum_{r=1}^n b_{rr})$, we have that $\boldA\models \mathcal{R}(b_{rs}')=0$ and $\|a_{rs}-b_{rs}\|<(n+1)\epsilon+\delta$ for all $r,s=1,\ldots,n$.  This suffices to establish the proof of the corollary.   
\end{proof}

\section{Computing $K_0$}\label{sec:K0}

The goal of this section is to establish the existence of a computable functor that takes a c.e. presentation of a \cstar-algebra $\boldA$ and produces a c.e. presentation of $K_0(\boldA)$.  

\subsection{Murray-von Neumann equivalence}
We begin by understanding the complexity of the Murray-von Neumann equivalence relation on projections in matrix amplifications of $\boldA$.  We let $P_n(\boldA)$ denote the set of all projections  in $M_n(\boldA)$, and set $P_{<\omega}(\boldA) = \bigcup_{n\in\mathbb{N}}P_n(\boldA)$.  For $p, q \in P_{<\omega}(\boldA)$, we set
\[p \oplus q = \left( 
\begin{array}{cc}
p & \zerom \\
\zerom & q\\
\end{array}
\right),\]
and note that $p \oplus q \in P_{<\omega}(\boldA)$.  Recall that $p \in P_m(\boldA)$ and $q \in P_n(\boldA)$ are \emph{Murray von-Neumann equivalent}, denoted $p \mvn q$, if there exists $v \in M_{m,n}(\boldA)$ so that $p = vv^*$ and $q = v^*v$.  We let $[p]_{\mvn}$ denote the Murray-von Neumann equivalence class of $p$.

\begin{theorem}\label{thm:mvn.ce}
Suppose $\boldA^\#$ is a c.e. presentation of $\boldA$. Then there is a  c.e.-open subset $U$ of $M_n(\boldA^\#) \times M_n(\boldA^\#)$ such that, for all $p,q\in P_n(\boldA)$, one has $(p,q)\in U$ if and only if $p\mvn q$.
\end{theorem}

\begin{proof}
Let $\mathcal{S}$ denote the set of all pairs $(B,B')$ of rational open balls of 
$M_n(\boldA^\#)$ for which there exists a sequence $B_0, \ldots, B_k$ of rational open balls of 
$M_{4n}(\boldA^\#)$ satisfying the following conditions:
\begin{enumerate}
    \item $\sup\{ \norm{a - b}\ : a \in B_j\ \text{ and }\ b \in B_{j+1}\} < 1$ for all $j=0,\ldots,k-1$.

    \item $B_j \cap P_{4n}(\boldA) \neq \emptyset$ for all $j=0,\ldots,k$.

    \item For all $a \in B$ and $b \in B_0$, 
    $\norm{((a\oplus\zerom_n)\oplus\zerom_{2n}) - b} < 1$.

    \item For all $a \in B_k$ and $b \in B'$, 
    $\norm{a - ((b\oplus\zerom_n)\oplus\zerom_{2n})} < 1$.
\end{enumerate}

By Corollary \ref{prop:c.e.clsd}, $P_{4n}(\boldA)$ is a c.e.-closed subset of $M_{4n}(\boldA^\#)$, whence it follows that
$\mathcal{S}$ is c.e.  Set $U = \bigcup\{ B \times B'\ :\ (B,B') \in \mathcal{S}\}$.  
If $(B,B') \in \mathcal{S}$, then $B \times B'$ is a c.e.-open subset of $M_n(\boldA^\#) \times M_n(\boldA^\#)$.  It follows that $U$ 
is a c.e.-open subset of $M_{4n}(\boldA^\#)$.  

We first show that, if $(p,q) \in U \cap (P_n(\boldA) \times P_n(\boldA))$, 
then $p \mvn q$.  To see this, suppose $(p,q) \in U \cap (P_n(\boldA) \times P_n(\boldA))$.  Take rational open balls $B,B',B_0$, $\ldots$, $B_k$
of $M_{4n}(\boldA^\#)$ satisfying the above conditions with 
$p \in B$ and $q \in B'$.
For each $j=1,\ldots,k$, fix $p_j \in B_j \cap P_{4n}(\boldA)$.  Then, 
$\norm{(p\oplus\zerom_n)\oplus\zerom_{2n} - p_0} < 1$, $\norm{(q\oplus\zerom_n)\oplus\zerom_{2n} - p_k} < 1$, and
$\norm{p_j - p_{j+1}} < 1$ for $j=0,\ldots,k-1$.  We may now use \cite[Propositions 2.2.4 and 2.2.7]{Rordam.Larsen.Laustsen.2000} to conclude that 
$$(p\oplus\zerom_n)\oplus\zerom_{2n} \mvn p_0 \mvn p_1\mvn \cdots \mvn p_k \mvn (q \oplus \zerom_n)\oplus\zerom_{2n}.$$  
By \cite[Proposition 2.3.2.i]{Rordam.Larsen.Laustsen.2000}, $p \mvn (p\oplus\zerom_n)\oplus\zerom_{2n}$ and 
$q \mvn (q\oplus\zerom_n)\oplus\zerom_{2n}$, whence the claim follows.

Conversely, suppose that $p,q \in P_n(\boldA)$ are such that 
$p \mvn q$.  We wish to show that $(p,q) \in U$.  Set $p' = (p \oplus \zerom_n) \oplus \zerom_{2n}$, and set 
$q' = (q \oplus \zerom_n) \oplus \zerom_{2n}$.  Note that $p'\mvn q'$.  By \cite[Proposition 2.2.8]{Rordam.Larsen.Laustsen.2000}, 
$p'$ is homotopic to $q'$ in $P_{4n}(\boldA)$; that is, there is a continuous function 
$f : [0,1] \rightarrow P_{4n}(\boldA)$ with $f(0) = p'$ and $f(1) = q'$.
Since $[0,1]$ is compact, $f$ is uniformly continuous.  Thus, there exist 
$p_0, \ldots, p_k \in \ran(f)$ so that $\norm{p_j - p_{j+1}} < 1$ for $j=1,\ldots,k-1$
and $\max\{\norm{p' - p_0}, \norm{q' - p_k}\} < 1$.  
The existence of the required rational open balls needed to show that $(p,q)\in U$ now follows from continuity.
\end{proof}

We note that the proof of Theorem \ref{thm:mvn.ce} is uniform.  Namely, it provides an algorithm 
that, given an index for a presentation $\boldA^\#$, computes an index for the c.e.-open set $U$ given in the conclusion of the theorem.

\begin{corollary}\label{cor:mvn.pred}
If $\boldA^\#$ is c.e., then there is a c.e. subset $Q \subseteq \N \times \N$ such that, 
for all $e,e' \in \N$, if $e$ and $e'$ are $\boldA^\#$-indices of projections $p$ and $p'$ respectively, then 
$p \mvn p'$ if and only if $Q(e,e')$.  Furthermore, an index of $Q$ can be computed from an index of 
$\boldA^\#$.
\end{corollary}

\begin{proof}
Let $\rho_n$ denote the $n$-th rational point of $M_n(\boldA^\#)$ and let $U$ be a c.e. open set of $M_n(\boldA^\#) \times M_n(\boldA^\#)$ that satisfies the conclusion of 
Theorem \ref{thm:mvn.ce}.  The desired set $Q$ is the set of those pairs $(e,e')\in \N^2$ for which there exist $s,k \in \N$ and $(B,B') \in U$ so that 
$\phi_{e,s}(k)\downarrow$, $\phi_{e',s}(k)\downarrow$, $B(\rho_{\phi_e(k)}; 2^{-k}) \subseteq B$, and 
$B(\rho_{\phi_{e'}(k)}; 2^{-k}) \subseteq B'$.  
\end{proof}

\subsection{The Murray-von Neumann semigroup}\label{subsec:MvN}

From now on, let $\mathcal{D}(\boldA)$ denote the set of Murray-von Neumann equivalence classes of elements of $\mathcal{P}_{<\omega}(\boldA)$.  The operation $\oplus$ induces an abelian semigroup operation on $\mathcal{D}(\boldA)$, which we continue to denote by $\oplus$.

Our goal in this subsection is to show how the presentation $\boldA^\#$ induces a presentation of $\mathcal{D}(\boldA)$.  
First, we must make precise what we mean by a presentation of $\mathcal{D}(\boldA)$.

Fix a countably infinite set $X = \{x_0, x_1, \ldots\}$ of indeterminates
and let $D_\omega$ and $F_\omega$ denote the free semigroup and free group generated by $X$ respectively. 
We implicitly view both $D_\omega$ and $F_\omega$ as equipped with a computable bijection with $\N$, 
whence it makes sense to speak of, say, c.e. and computable subsets of $D_\omega$ and $F_\omega$.
The following definitions are taken from algorithmic group theory.

\begin{definition}\label{def:semig.pres}
 If $S$ is a semigroup, then a \emph{presentation of $S$} is a pair $(S, \nu)$, where $\nu$ is a semigroup epimorphism of $D_\omega$ onto $S$. 
If $G$ is a group, then a \emph{presentation of $G$} is a pair $(G, \nu)$, where $\nu$ is a group epimorphism of $F_\omega$ onto $S$.
\end{definition}

As with \cstar algebras, if $S$ is a semigroup or a group, then we use expressions such as $S^\#$ and $S^\dagger$ to denote
presentations of $S$.

We introduce some terminology associated with presentations of groups and semigroups.
Suppose $S$ is a semigroup or group and $S^\# = (S, \nu)$ is a presentation of $S$.  
If $\nu(w) = a$, then we call $w$ an \emph{$S^\#$-label} of $a$.  The \emph{kernel} of $S^\#$, denoted $\ker(\nu)$, is the set 
of all pairs $(w,w')$ so that $w$ and $w'$ are $S^\#$-labels of the same element of $S$.

\begin{definition}\label{def:pres.etc}
Suppose $S^\#$ is a presentation of a group or semigroup.
We say that $S^\#$ is \emph{c.e.} (resp. \emph{computable}) if its kernel is c.e. (resp. computable).
\end{definition}

We remark that in algorithmic group theory, c.e. presentations are called \emph{recursive} and 
computable presentations are called \emph{decidable}.

\begin{definition}\label{def:semig.map.comp}
Suppose $S_0^\#$ and $S_1^\#$ are (semi)group presentations.  A function $f : S_0 \rightarrow S_1$ is a \emph{computable map from $S_0^\#$ to $S_1^\#$} if there is a computable map 
$F$ so that, for every $a\in S_0$ and every $S_0^\#$-label $w$ of $a$, $F(w)$ is an $S_1^\#$-label of $f(a)$. 
\end{definition}

To be clear, in the previous definition, $S_0$ could be a semigroup while $S_1$ is a group (or vice-versa)

\begin{definition}\label{def:supports}
Suppose that $\boldA^\#$ is a presentation of $\boldA$ and $\mathcal{D}(\boldA)^\#$ is a presentation of $\mathcal{D}(\boldA)$. We say that $\boldA^\#$ \emph{computably supports} $\mathcal{D}(\boldA)^\#$ if there is an algorithm that, given $w \in D_\omega$, computes $n,e \in \N$ so that $e$ is an $M_n(\boldA^\#)$-index of
a projection $p \in M_n(\boldA)$ for which $w$ is a $\mathcal{D}(\boldA)^\#$-label of $[p]_{\mvn}$.
\end{definition}

\begin{proposition}\label{prop:support.unique}
If $\boldA^\#$ is c.e., then, up to computable isomorphism, there is at most one 
c.e. presentation of $\mathcal{D}(\boldA)$ that is computably supported by $\boldA^\#$.
\end{proposition}

\begin{proof}
Suppose $\mathcal{D}(\boldA)^\#$ and $\mathcal{D}(\boldA)^\dagger$ are 
c.e. presentations of $\mathcal{D}(\boldA)$ that are computably supported by $\boldA^\#$.  
We define a function $F : D_\omega \rightarrow D_\omega$ as follows.  
Given $w \in D_\omega$, we compute $e_w, n_w$ so that $e_w$ is an $M_{n_w}(\boldA)^\#$-index of a projection 
$p_w \in P_{n_w}(\boldA)$ for which $w$ is a $\mathcal{D}(\boldA)^\#$-label of $[p_w]_{\mvn}$.
Similarly, for each $v \in D_\omega$, we
compute $n_v,e_v \in \N$ so that $e_v$ is an $M_{n_v}(\boldA)^\#$-index of a 
projection $p_v^\dagger$ for which $v$ is a $\mathcal{D}(\boldA)^\dagger$-label of $[p_v^\dagger]_{\mvn}$.
By Corollary \ref{cor:mvn.pred}, we can effectively search for a $v \in D_\omega$ so that 
$p_w \mvn p_v^\dagger$ and we then set $F(w)$ to be this $v$.  Thus, $F$ is computable.  Moreover, by construction, $F$ witnesses that the identity map on $\mathcal{D}(\boldA)$ is a computable isomorphism of $\mathcal{D}(\boldA)^\#$ to $\mathcal{D}(\boldA)^\dagger$.
\end{proof}

\begin{proposition}\label{prop:seq.proj}
If $\boldA^\#$ is c.e. and $n$ is a positive integer, 
then there is an $M_n(\boldA^\#)$-computable sequence $(p_k)_{k \in \N}$
of projections such that every projection of $M_n(\boldA)$ is Murray-von Neumann equivalent to some $p_k$.
\end{proposition}

\begin{proof}
Let $\mathcal{S}_n$ denote the set of all rational open balls $B$ of $M_n(\boldA^\#)$ for which 
$B \cap P_n(\boldA) \neq \emptyset$.  By Corollary \ref{prop:c.e.clsd}, $\mathcal{S}_n$ is c.e. 
Fix an effective one-to-one enumeration 
$(B(a_j;r_j))_{j \in \N}$ of $\{B \in \mathcal{S}_n\ :\ \diam(B) < 2^{-1}\}$ and set $B_j = B(a_j; r_j)$. 

By Lemma \ref{intersection}, for each $k\in \N$, we may compute an $M_n(\boldA^\#)$-index for a projection $p_k\in B_k$.  
Now given $p\in \mathcal{P}_n(\boldA)$, choose $k\in \N$ such that $p\in B_k$; it follows that $\|p-p_k\|<2r_k<1$, whence $p\mvn p_k$. 
\end{proof}

The proof of Proposition \ref{prop:seq.proj} is uniform.  Namely, it constructs an algorithm such that, given $n$ and an index for $A^\#$, 
computes an $M_n(\boldA^\#)$-index for a sequence as in the conclusion of the lemma.  

\begin{theorem}\label{thm:pres.DA}
If $\boldA^\#$ is c.e., then there is a c.e. presentation $\mathcal{D}(\boldA^\#)$ of $\mathcal{D}(\boldA)$ that is supported by
$\boldA^\#$.
\end{theorem}

\begin{proof}
Suppose $\boldA^\#$ is c.e.  By the uniformity of Proposition \ref{prop:seq.proj}, 
there is an array $(p_{n,k})_{n,k \in \N}$ that satisfies the following conditions.
\begin{enumerate}
    \item $(p_{n,k})_{k \in \N}$ is a computable sequence of projections in $M_n(\boldA^\#)$, uniformly in $n$.

    \item For every projection $p \in P_n(\boldA)$, there exists $k \in \N$ so that $p \mvn p_{n,k}$.
\end{enumerate}
We define the semigroup epimorphism $\phi:D_\omega\to \mathcal{D}(\boldA)$ yielding the presentation $\mathcal{D}(\boldA^\#)$ as follows.  First, set $q_{\langle m,k \rangle} = p_{m,k}$ for all $m,k \in \N$.  
Given $w = \prod_{j \in F} x_j \in D_\omega$, where $F \subseteq \N$ is finite, set $\phi_0(w) =  \bigoplus_{j \in F} q_j\in \mathcal{P}_{<\omega}(\boldA)$ and let $\phi(w) = [\phi_0(w)]_{\mvn}\in \mathcal{D}(\boldA)$.  Note that $\phi$ is indeed a semigroup epimorphism.  
Finally, we set $\mathcal{D}(\boldA^\#) = (\mathcal{D}(\boldA), \phi)$.  

It is straightforward to see that $\mathcal{D}(\boldA^\#)$ is computably supported by $\boldA^\#$.
We now show $\mathcal{D}(\boldA^\#)$ is c.e.  Given $w,w' \in D_\omega$, set:
\begin{eqnarray*}
n & = & n_w + n_{w'}\\
q & = & \phi_0(w) \oplus \zerom_{n_{w'}}\\
q' & = & \phi_0(w') \oplus \zerom_{n_w}.
\end{eqnarray*}
Thus $(w,w') \in \ker(\mathcal{D}(\boldA^\#) )$ if and only if $q \mvn q'$.
In addition, $M_n(\boldA^\#)$-indices of $q$ and $q'$ can be computed from $w$, $w'$.  
It follows from the uniformity of Corollary \ref{cor:mvn.pred} that $\ker(\mathcal{D}(\boldA^\#))$ is c.e. 
\end{proof}

Once again, we remark that the proof Theorem \ref{thm:pres.DA} is uniform.


If $\psi : \boldA_0 \rightarrow \boldA_1$ is a $*$-homomorphism, we let $\psi_n:M_n(\boldA_0)\to M_n(\boldA_1)$ denote the matrix amplification of $\psi$.  We note that if $\psi$ is also a computable map from $\boldA_0^\#$ to $\boldA_1^\#$, then each 
$\psi_n$ is a computable map from $M_n(\boldA_0)^\#$ to $M_n(\boldA_1)^\#$.  We also let $\mathcal{D}(\psi):\mathcal{D}(\boldA_0)\to \mathcal{D}(\boldA_1)$ be the semigroup homomorphism defined by
$\mathcal{D}(\psi)([p]_{\mvn}) = [\psi_n(p)]_{\mvn}$ whenever $p \in P_n(\boldA)$.

\begin{theorem}\label{thm:D(psi).comp}
Suppose $\boldA_0^\#$ and $\boldA_1^\#$ are c.e. and $\psi$ is a computable 
$\star$-homomorphism from $\boldA_0^\#$ to $\boldA_1^\#$.  Then $\mathcal{D}(\psi)$ is a computable
map from $\mathcal{D}(\boldA_0^\#)$ to $\mathcal{D}(\boldA_1^\#)$.  
\end{theorem}

\begin{proof}
Given $w \in D_\omega$, do the following.  First, compute $n$ and an $M_n(\boldA_0)^\#$-index $e$ of a projection 
$p_0 \in M_n(\boldA_0)^\#$ so that $w$ is a $\mathcal{D}(\boldA_0)^\#$-label of $p_0$.  
Since $\psi_n$ is computable, we can then compute an $M_n(\boldA_1)^\#$-index $e_1$ of 
$\psi_n(p_0)$.  Since $\boldA_1^\#$ supports $\mathcal{D}(\boldA_1^\#)$, by 
Theorem \ref{thm:mvn.ce}, we can then search for a $\mathcal{D}(\boldA_1^\#)$-label of $[\psi_n(p_0)]_{\mvn} = \mathcal{D}(\psi)([p_0]_{\mvn})$. 
\end{proof}

Once again, the statement of Theorem \ref{thm:D(psi).comp} holds uniformly.

As previously remarked, we view presentations of unital \cstar-algebras as forming a category; we also consider the full subcategories of c.e. and computable presentations of \cstar-algebras.  Similarly, we have the category of (semi)group presentations, where the morphisms are the computable homomorphisms; we analogously can consider the full subcategories of c.e. and computable presentations of (semi)groups.  Our work in this section can be summarized as follows.  

\begin{corollary}\label{cor:D.comp.funct}
$\mathcal{D}$ is a computable functor from the category of c.e. presentations of unital \cstar-algebras to the category of c.e. semigroup presentations.
\end{corollary}

In this corollary, when we say that $\mathcal{D}$ is a computable functor, we mean that there is an algorithm that, given an index 
of an object or morphism, computes an index for its $\mathcal{D}$-image.  We will use the word computable in the same way for the other functors defined later in this paper.

\subsection{The $K_0$ group}

For a unital \cstar-algebra $\boldA$, $K_0(\boldA)$ is defined by taking the Grothendieck group of the Murray-von Neumann semigroup $\mathcal{D}(\boldA)$.  Consequently, we begin this section by recalling the definition of the Grothendieck functor $\mathcal{G}$ from the category of abelian semigroups to the category of 
abelian groups.    

Fix an abelian semigroup $S$ and consider the equivalence relation $\sim$ on $S\times S$ given by $(a,b) \sim (c,d)$ if and only if there exists $z \in S$ so that 
$a + d + z = c + b + z$.  Let $\mathcal{G}(S)=(S\times S)/\sim$ denote the quotient space.  Then the semigroup operation on $S\times S$ (defined coordinatewise) induces an abelian group operation on $\mathcal{G}(S)$, where the additive inverse of the equivalence class of $(a,b)$ is the equivalence class of $(b,a)$.  We call $\mathcal{G}(S)$ the \emph{Grothendieck group of the semigroup $S$}.  We let $\pi:S\times S \to \mathcal{G}(S)$ denote the natural quotient map.  There is also a map $\gamma_S : S \rightarrow \mathcal{G}(S)$ given by setting $\gamma_S(a)=\pi(a,b)$, where $b\in S$ is any fixed element.  (The definition of $\gamma_S$ does not depend on the choice of $b\in S$.)  The reader should be warned that $\gamma_S$ is not injective in general; in fact, $\gamma_S$ is injective precisely when $S$ is a cancellative semigroup.  Every element of $\mathcal{G}(S)$ can be written as a difference $\gamma_S(a)-\gamma_S(b)$ for some $a,b\in S$.

The Grothendieck group $\mathcal{G}(S)$ satisfies the following universal property:  if $\phi : S \rightarrow H$ is a semigroup morphism from an abelian semigroup to an abelian group, then there is a unique group homomorphism 
$\psi_\phi:\mathcal{G}(S)\to H$ so that $\psi_\phi \circ \gamma_S = \phi$.  A consequence of this fact is that $\mathcal{G}$ is actually a functor, that is, given a semigroup morphism $\phi:S_0\to S_1$, there is a unique group morphism $\mathcal{G}(\phi):\mathcal{G}(S_0)\to \mathcal{G}(S_1)$ such that $\mathcal{G}(\phi)\circ \gamma_{S_0}=\gamma_{S_1}\circ \phi$.

We now seek to effective this construction.

\begin{definition}\label{def:univ}
Let $S^\#$ be a presentation of an abelian semigroup and let $\mathcal{G}(S)^\#$ be a presentation of $\mathcal{G}(S)$.
We say that $\mathcal{G}(S)^\#$ is \emph{$S^\#$-universal} if it satisfies the following conditions:
\begin{enumerate}
    \item $\gamma_S$ is a computable map from $S^\#$ to $\mathcal{G}(S)^\#$.

    \item Whenever $H^\#$ is a c.e. presentation of an abelian group and $\phi$ is a computable semigroup morphism from 
    $S^\#$ to $H^\#$, then $\psi_\phi$ is a computable map from $\mathcal{G}(S)^\#$ to $H^\#$ and, moreover, an 
    index of $\psi_\phi$ can be computed from an index of $\phi$.
\end{enumerate}
\end{definition}


It is straightforward to show that if there is an $S^\#$-universal presentation of $\mathcal{G}(S)$, then this presentation is unique 
up to computable isomorphism.

\begin{definition}\label{def:c.e.grp.pres}
If $S^\#$ is a (semi)group presentation, then $X \subseteq S$ is said to be a \emph{c.e. subset of $S^\#$} if 
the set of all $S^\#$-labels of elements of $S$ is c.e.
\end{definition}

\begin{proposition}\label{prop:univ.exist}
If $S^\#$ is a c.e. presentation of an abelian semigroup, then there is a (necessarily unique) c.e. presentation
$\mathcal{G}(S^\#)$ of $\mathcal{G}(S)$ that is $S^\#$-universal.
\end{proposition}

\begin{proof}
We first claim that there is a c.e. presentation $(S \times S)^\#$ of the semigroup $S\times S$ that satisfies the following conditions.

\begin{enumerate}
    \item The projection functions are computable maps from $(S \times S)^\#$ to $S^\#$.

    \item For each $a \in S$, the injection maps $b \mapsto (a,b)$ and $b \mapsto (b,a)$ 
    are computable maps from $S^\#$ to $(S \times S)^\#$.
\end{enumerate}

To see this, write $S^\# = (S, \nu)$.  Fix a computable bijection $\tau = (\tau_0, \tau_1)$ of $\N$ onto 
$(\N^{<\omega} \times \N^{< \omega}) \setminus \{(\emptyset, \emptyset)\}$.
For each $n \in \N$ and $k \in \{0,1\}$, let $\sigma_k(x_n) = \prod_{j < |\tau_k(n)|} x_{\tau_k(j)}$.  
Set $\sigma = (\sigma_0, \sigma_1)$.
Then, for each $w \in D_\omega$, let $\phi(w) = \prod_{j < |w|} \sigma(w_j)$.   
It follows that $(D_\omega \times D_\omega, \phi)$ is a computable presentation of 
$D_\omega \times D_\omega$.  
Let $\phi_0$, $\phi_1$ be the component maps of $\phi$.  
Set $\nu_1(w) = (\nu\phi_0(w), \nu\phi_1(w))$, and set $(S \times S)^\# = (S \times S, \nu_1)$.  
It follows that $(S \times S)^\#$ has the required properties.  Note also that the Grothendieck equivalence relation $\sim$ is a c.e. subset of $(S \times S)^\#$. 

We next claim that there is a c.e. presentation 
$\mathcal{G}(S^\#)$ so that the canonical epimorphism $\pi:S\times S\to \mathcal{G}(S)$ is a computable map from 
$(S \times S)^\#$ to $\mathcal{G}(S^\#)$.  Let $\nu_1$ be the labeling map of $(S \times S)^\#$ and set $\nu_2 = \pi\nu_1$.  Extend $\nu_2$ to $F_\omega$ by defining 
$\nu_2(x_n^{-1})=\nu_2(x_n)^{-1}$ and $\nu_2(\unit) = \nu_2(x_0x_0^{-1})$, and set $\mathcal{G}(S^\#) = (\mathcal{G}, \nu_2)$. 
Then $\mathcal{G}(S^\#)$ has the required properties.  
Note also that, since $\sim$ is a c.e. subset of $(S\times S)^\#$, $\mathcal{G}(S)^\#$ is a c.e. presentation of $\mathcal{G}(S)$.

We now argue for the $S^\#$-universality of $\mathcal{G}(S^\#)$. 
Let $a_0$ be the element of $S$ whose $S^\#$-label is the indeterminate $x_0$. We have, 
$\gamma_S(a) = [(a,a_0)]_{\sim}$ for all $s \in S$.  Thus, by the computability of the 
injections and the canonical epimorphism, $\gamma_S$ is a computable map from $S^\#$ to $\mathcal{G}(S^\#)$.  
Suppose now that $H^\#$ is a c.e. presentation of an abelian group, and suppose $\phi$ is a computable
semigroup morphism from $S^\#$ to $H^\#$.
Suppose $g_0 \in F_\omega$ is a $\mathcal{G}(S^\#)$-label of $[(a,b)]_\sim$.  
By using the canonical epimorphism, we can compute a $(S \times S)^\#$-label $w_1$ 
of a pair $(a',b')$ so that 
$(a',b') \sim (a,b)$.  By using the projection maps, we can compute 
$S^\#$-labels of $a'$ and $b'$.
We then compute an $H^\#$-label of $\phi(a') - \phi(b') = \psi_\phi([(a,b)]_\sim)$.
\end{proof}

The statement of Proposition \ref{prop:univ.exist} holds uniformly in the sense that, from an index of $S^\#$, one can compute
indices for $\mathcal{G}(S^\#)$ and $\gamma_S$.



\begin{corollary}\label{cor:groth.comp}
$\mathcal{G}$ is a computable functor from the category of c.e. presentations of abelian semigroups to the category of 
c.e. presentations of abelian groups.
\end{corollary}

\begin{proof}
It only remains to show that $\mathcal{G}$ uniformly transforms computable semigroup homomorphisms into 
comptuable group homomorphisms.  Suppose 
$\phi$ is a computable semigroup homomorphism from $S_0^\#$ to $S_1^\#$.  
From a $\mathcal{G}(S_0^\#)$-label of $g \in \mathcal{G}(S_0)$, one can compute $S_0^\#$-labels of $a,b \in S_0$ so that 
$g = \gamma_{S_0}(a) - \gamma_{S_0}(b)$.  Consequently, we can then compute a $\mathcal{G}(S_1^\#)$-label of 
$\mathcal{G}(\phi)(g) = \gamma_{S_1}(\phi(a)) - \gamma_{S_1}(\phi(b))$. 
\end{proof}

Given a c.e. presentation $\boldA^\#$, we set $K_0(\boldA^\#) = (\mathcal{G}\circ \mathcal{D})(\boldA^\#)$.

\begin{corollary}\label{cor:K0.comp.func}
$K_0$ is a computable functor from the category of c.e. presentations of unital \cstar-algebras to the 
category of c.e. presentations of abelian groups.
\end{corollary}

The previous corollary raises the following natural question.

\begin{question}
If $\boldA^\#$ is computable, is $K_0(\boldA^\#)$ also computable?  More generally, if $\boldA$ is computably presentable, is $K_0(\boldA)$ also computably presentable?
\end{question}

We will see that in the case of UHF algebras, the previous question has a positive answer.  In order to prove that result, we will require the following:

\begin{proposition}\label{Grothendieckcomputable}
Suppose that $S^\#$ is a computable presentation of a \emph{cancellative} abelian semigroup $S$.  Then $\mathcal G(S^\#)$ is also computable.
\end{proposition}

\begin{proof}
Suppose $x_1,x_2\in F_\omega$ are $\mathcal G(S^\#)$-labels for $a_1,a_2\in \mathcal G(S)$.  Since $\gamma_S$ is a computable map from $S^\#$ to $\mathcal G(S^\#)$, we can effectively search for $y_{i,j}\in D_\omega$ which are $S^\#$-labels for $b_{i,j}\in \mathcal G(S)$ satisfying $a_i=\gamma_S(b_{i,1})-\gamma_S(b_{i,2})$.  We then have that $a_1=a_2$ if and only if $\gamma_S(b_{1,1})+\gamma_S(b_{2,2})=\gamma_S(b_{1,2})+\gamma_S(b_{2,1})$, which, since $S$ is cancellative, is equivalent to $b_{1,1}+b_{2,2}=b_{1,2}+b_{2,1}$; since $S^\#$ is computable, this can indeed be decided.
\end{proof}

Recall that an \emph{ordered abelian group} is a pair $(G,G^+)$, where $G$ is a group and $G^+$ is a subset of $G$, called the \emph{positive cone of $G$}, that satisfy the following three properties.

\begin{enumerate}
    \item $G^++G^+\subseteq G^+$.
    \item $G^+\cap (-G^+)=\{0\}$.
    \item $G=G^+-G^+$.
\end{enumerate}

If $\boldA$ is stably finite, then $K_0(\boldA)$ is an ordered abelian group with positive cone $K_0(\boldA)^+=\{\gamma_{\mathcal{D}(\boldA)}([p]_{\mvn}) \ : \ p\in \mathcal{P}_{<\omega}(\boldA)\}$.

\begin{proposition}\label{prop:cone.ce}
If $\boldA$ is stably finite and $\boldA^\#$ is c.e., 
then $K_0(\boldA)^+$ is a c.e. subset of $K_0(\boldA^\#)$.
\end{proposition}

\begin{proof}
For simplicity, set $S = \mathcal{D}(\boldA)$, so $K_0(\boldA)^+ = \ran(\gamma_S)$.
Write $\mathcal{D}(\boldA^\#) = (\mathcal{D}(\boldA), \nu_0)$ and 
$K_0(\boldA^\#) = (K_0(\boldA), \nu_1)$.  Then by the computable universality of $K_0(\boldA^\#)$,
there is a computable
$F : D_\omega \rightarrow F_\omega$ so that $\gamma_S \nu_0 = \nu_1 F$.
A direct computation shows that $$\nu_1^{-1}[\ran(\gamma_S)] =\{w\in F_\omega \ : \ (w,F(w'))\in \ker(\nu_1) \text{ for some }w'\in F_\omega\}.$$
Since $\ker(\nu_1)$ is c.e. and $F$ is computable, we see that $\nu_1^{-1}[\ran(\gamma_S)]$ is 
c.e., whence $K_0(\boldA)^+$ is a c.e. subset of $K_0(\boldA^\#)$.
\end{proof}

\begin{corollary}\label{linearcomputable}
Suppose that $\boldA$ is stably finite and $K_0(\boldA)$ is linearly ordered, that is, $K_0(\boldA)=K_0(\boldA)^+\cup (-K_0(\boldA)^+)$.  Then for any c.e. presentation $\boldA^\#$ of $\boldA$, $K_0(\boldA)^+$ is a computable subset of $K_0(\boldA^\#)$.
\end{corollary}

\begin{proof}
This follows from Proposition \ref{prop:cone.ce} together with the easy observation that, for any presentation $G^\#$ of of an abelian group, the map $x\mapsto -x$ is a computable map from $G^\#$ to itself.
\end{proof}

\subsection{Non-unital \cstar-algebras}

In this section, we deviate from our convention and do not assume that $\boldA$ is necessarily unital.  Recall that the \emph{unitization} 
$\widetilde{\boldA}$ of $\boldA$ is the \cstar-algebra with underlying set $\widetilde{\boldA}=\{(a,\alpha) \ : \ a\in A, \alpha\in 
\mathbb{C}\}$, with addition and scalar multiplication given coordinatewise, and multiplication given by
$(a,\alpha)\cdot (b,\beta)=(ab+\beta a+\alpha b,\alpha\beta)$.  Equipped with these operations, $\widetilde{\boldA}$ is a unital *-algebra
with unit $(0,1)$.  One has *-algebra homomorphisms $\iota:\boldA\to \widetilde{\boldA}$ and $\pi:\widetilde{\boldA}\to \mathbb{C}$ given by
$\iota(a)=(a,0)$ and $\pi(a,\alpha)=\alpha$.  Note that $\iota(\boldA)$ is an ideal of $\widetilde{\boldA}$, which we usually identify with 
$\boldA$ itself. $\widetilde{\boldA}$ is in fact a \cstar-algebra under the norm 
$\|x\|_{\widetilde{\boldA}}=\max(|\pi(x)|,\sup\{\|ax\|_{\boldA} \ : \ a\in A, \|a\|\leq 1\})$.  
Note that when $\boldA$ is unital, then $\widetilde{\boldA}$ is isomorphic to $\boldA\oplus
\mathbb{C}$.

If $\boldA^\#$ is a presentation of $\boldA$, then one obtains a presentation $\widetilde{\boldA}^\#$ on $\widetilde{\boldA}$ by declaring that the $\langle m,k\rangle$-th special point of $\widetilde{\boldA}^\#$ is the pair $(a,\alpha)$, where $a$ is the $m$-special point of $\boldA^\#$ and $\alpha$ is the $k$-th special point of the standard presentation of $\mathbb{C}$.  It follows immediately that if $\boldA^\#$ is a left-c.e. presentation, then so is $\widetilde{\boldA}^\#$.

An alternate description of the unitization in terms of universal \cstar-algebras will prove useful.  Namely, $\widetilde{\boldA}$ is the universal \cstar-algebra generated by $\boldA$ and an additional indeterminate $x$ subject to all of the relations satisfied by $\boldA$ together with the relations stating that $x$ is a contraction that serves as a unit for the algebra.  In particular, if $\boldA^\#$ is a presentation of $\boldA$, then in the previous sentence, one may replace the statement ``all of the relations satisfied by $\boldA$'' by the statement ``all rational relations $\|p(\vec a)\|\leq r$'' satisfied by the special points of $\boldA^\#$.  In this way, one can equip $\widetilde{\boldA}$ with the corresponding universal presentation $\widetilde{\boldA}^{\#,u}$ (see \cite[Subsection 2.2]{Goldbring.2024+}), which is easily seen to be computably isomorphic to the previously defined presentation $\widetilde{\boldA}^\#$. 

\begin{proposition}
If $\boldA^\#$ is a c.e. (resp. computable) presentation of $\boldA$, then the induced presentation on $\widetilde{\boldA}^\#$ is also c.e. (resp. computable).  Moreover, an index for $\widetilde{\boldA}^\#$ can be computed from an index for $\boldA^\#$. 
\end{proposition}

\begin{proof}
The statement for c.e. presentations follows immediately from \cite[Theorem 3.3]{fox2022computable}.  The statement for computable presentations follows from the  statement for right-c.e. presentations and the corresponding statement for left-c.e. presentations mentioned above.
\end{proof}

Given a *-algebra homomorphism $\phi:\boldA_0\to \boldA_1$, we obtain a *-algebra homomorphism $\widetilde{\phi}:\widetilde{\boldA_0}\to \widetilde{\boldA_1}$ given by $\widetilde{\phi}(a,\alpha)=(\phi(a),\alpha)$.  The following fact is clear.

\begin{proposition}
If $\phi:\boldA_0^\#\to \boldA_1^\#$ is a computable *-algebra homomorphism, then so is $\widetilde{\phi}:\widetilde{\boldA_0}^\#\to \widetilde{\boldA_1}^\#$.
\end{proposition}

\begin{corollary}
The unitization functor from the category of c.e. (resp. computable) presentations of \cstar-algebras to the category of c.e. (resp. computable presentations) of unital \cstar-algebras is computable.
\end{corollary}

For $\boldA$ not necessarily unital, $K_0(\boldA)$ is defined to be the kernel of the map $K_0(\pi):K_0(\widetilde{\boldA})\to K_0(\mathbb{C})$.  In order to equip $K_0(\boldA)$ with a c.e. presentation, we need a couple of preliminary lemmas.  It is quite possible that these lemmas are folklore in the algorithmic group theory community, but since we could not find exact statements of them in the literature, we provide proofs for the convenience of the reader.

\begin{lemma}
If $\phi:G_0^\#\to G_1^\#$ is a computable map of group presentations and $G_1^\#$ is c.e., then $\ker(\phi)$ is a c.e. subset of $G_0^\#$.  Moreover, an index for $\ker(\phi)$ can be computed from indices for $\phi$ and $G_1^\#$.
\end{lemma}

\begin{proof}
Suppose that $G_i^\#=(G_i,\nu_i)$, $i=0,1$, and fix $w'\in F_\omega$ such that $\nu_1(w)=0$.  Take also a computable function $f:F_\omega\to F_\omega$ witnessing that $\phi$ is computable.  Then $\nu_0^{-1}[\ker(\phi)]=\{w\in F_\omega \ : \ (f(w),w')\in \ker(\nu_1)\}$.  Since $f$ is computable and $\ker(\nu_1)$ is c.e., we see that $\nu_0^{-1}[\ker(\phi)]$ is a c.e. subset of $F_\omega$, whence $\ker(\phi)$ is a c.e. subset of $G_0^\#$.
\end{proof}

\begin{lemma}
If $G^\#$ is a c.e. group presentation and $H$ is a c.e. subgroup of $G^\#$, then there is a c.e. presentation $H^\#$ of $H$ such that the inclusion map is a computable map $H^\#\to G^\#$.  Moreover, an index for $H^\#$ can be computed from an index for $G^\#$ and an index for $H$ as a c.e. set.
\end{lemma}

\begin{proof}
Let $U$ be the set of all $G^\#$-labels of elements of $S$.  Then $U$ is c.e. and $U$ is a subgroup of $F_\omega$. 
Let $(w_n)_{n \in \N}$ be an effective enumeration of $U$.  Set $\nu_0(x_n) = w_n$, extend $\nu_0$ to an epimorphism of 
$F_\omega$ onto $U$, and set $U^\# = (U, \nu_0)$.  Then $U^\#$ is a c.e. presentation of $U$.  Define $H^\#$ by declaring 
$w$ to be an $H^\#$-label of $a \in H$ if $w$ is a $U^\#$-label of a $G^\#$-label of $a$.  It follows that 
$H^\#$ has all the required properties.
\end{proof}

Given a c.e. presentation $\boldA^\#$ of $\boldA$, define $K_0(\boldA^\#)$ to be the c.e. presentation of $K_0(\boldA)=\ker(\pi)$ induced by the c.e. presentation of $K_0(\widetilde{\boldA}^\#)$, where $\widetilde{\boldA}^\#$ is the c.e. presentation of $\widetilde{\boldA}$ induced by the presentation $\boldA^\#$.

\begin{lemma}
Suppose that $\phi:\boldA_0^\#\to \boldA_1^\#$ is a computable *-homomorphism.  Then the induced map $K_0(\phi):K_0(\boldA_0)\to K_0(\boldA_1)$ is a computable map from $K_0(\boldA_0^\#)$ to $K_0(\boldA_1^\#)$.  Moreover, an index for $K_0(\phi)$ can be computed from indices for $\phi$, $A_0^\#$, and $A_1^\#$.
\end{lemma}

\begin{proof}
Note that $K_0(\phi)$ is the restriction of $K_0(\widetilde{\phi}):K_0(\widetilde{\boldA_0})\to K_0(\widetilde{\boldA_1})$ to $K_0(\boldA_0)$.
The inclusion map is a computable map from $K_0(\boldA_i^\#)$ to  $K_0(\widetilde{\boldA_i}^\#)$.  The result then follows from the fact that 
$K_0(\phi)$ is a computable map of $K_0(\widetilde{\boldA_0}^\#)$ to $K_0(\widetilde{\boldA_1}^\#)$.  
Since all of the assertions in this proof 
have already been show to be uniform, so is the conclusion of the lemma.
\end{proof}

Putting all of the pieces from this subsection together, we obtain the following.

\begin{corollary}
$K_0$ is a computable functor from the category of c.e. presentations of (not necessarily unital) \cstar algebras to the category of c.e. presentations of groups.
\end{corollary}

\section{Computing $K_1$}\label{sec:K1}

Having established that c.e. presentations of $\boldA$ give rise to c.e. presentations of $K_0(\boldA)$, it is natural to turn our attention next to $K_1(\boldA)$.  Fortunately, our work on $K_0$ will allow us to accomplish this quite quickly.  In this section, we also abandon our convention that all \cstar-algebras are assumed to be unital.

Recall that for any \cstar-algebra $\boldA$, the \emph{suspension of $\boldA$} is the \cstar-algebra $C_0(0,1)\otimes \boldA$; we denote
this algebra by $S\boldA$.  It follows that 
$S\boldA \cong \{f\in C([0,1],\boldA) \ : \ f(0)=f(1)=0\}$ where the isomorphism is given by the mapping
$f\otimes a\mapsto (t\mapsto f(t)a)$.  Note that $S\boldA$ is not unital.

\begin{lemma}\label{inducedsuspension}
If $\boldA^\#$ is a c.e. (resp. computable) presentation, then there is an induced presentation $S\boldA^\#$ of $S\boldA$ that is also c.e. (resp. computable).  Moreover, an index for $S\boldA^\#$ can be computed from an index for $\boldA^\#$.
\end{lemma}

\begin{proof}
    We begin by noting that $C_0(0,1)$ has a computable presentation $C_0(0,1)^\#$, namely the ones whose special points are the rational polygonal curves $\gamma : [0,1] \rightarrow \C$ so that 
$\gamma(0) = \gamma(1) = 0$.  Equip $S\boldA^\#$ with the presentation whose $\langle m,k\rangle$th special point is $f_m\otimes a_k$, where $f_m$ and $a_k$ are the $m$th rational point of $C_0(0,1)^\#$ and $k$th rational point of $\boldA^\#$ respectively.  By \cite[Lemmas 2.4 and 2.5]{Goldbring.2024+}, we have that $S\boldA^\#$ is c.e. when $\boldA^\#$ is c.e.  If $\boldA^\#$ is actually computable, then $S\boldA^\#$ is also left-c.e. (and hence computable) using the formula $$\left\|\sum_i f_i\otimes a_i\right\|=\sup_{t\in (0,1)}\left\|\sum_i f_i(t)a_i\right\|$$ and the fact that $C_0(0,1)^\#$ is \emph{evaluative} in the sense of \cite{mcnicholl2024evaluative}, that is, that the map $(f,t)\mapsto (f,t):(C_0(0,1)\times (0,1))^\#\to \mathbb{C}^\#$ is computable, where $(0,1)$ and $\mathbb{C}$ are given their standard presentations.
\end{proof}

The suspension operation is actually an endofunctor on the category of \cstar-algebras:  given a *-homomorphism $\phi:\boldA_0\to \boldA_1$, one gets another *-homomorphism $S\phi:S\boldA_0\to S\boldA_1$ given by $S(\phi)(\sum_i f_i\otimes a_i)=\sum_i f_i\otimes \phi(a_i)$.  The following lemma is immediate:

\begin{lemma}
If $\phi:\boldA_0^\#\to \boldA_1^\#$ is computable, then so is $S\phi:S\boldA_0^\#\to \boldA_1^\#$.  Moreover, an index for $S\phi$ can be computed from indices for $\boldA_0^\#,\boldA_1^\#$, and $\phi$.
\end{lemma}

We let $S$ denote the endofunctor on both the category of computable presentations of \cstar algebras and its full subcategory of c.e. presentations, which associates to each c.e. (resp. computable) presentation $\boldA^\#$ the c.e. (resp. computable) presentation $S\boldA^\#$ as in Lemma \ref{inducedsuspension}.  We have just shown the following.

\begin{corollary}
$S$ is a computable functor.
\end{corollary}

The reason for introducing the suspension functor is that one may define $K_1(\boldA)$ as $K_1(\boldA)=K_0(S\boldA)$.  Consequently, we may define the functor $K_1=K_0\circ S$.  As a composition of two computable functors is computable, we obtain the following.

\begin{theorem}
The functor $K_1$ is a computable functor from the category of c.e. presentations of \cstar-algebras to the category of c.e. presentations of groups.
\end{theorem}

Our definition of $K_1(\boldA)$ is perhaps not the usual one, but rather is equivalent to the usual definition only via a nontrivial theorem 
(see \cite[Theorem 10.1.3]{Rordam.Larsen.Laustsen.2000}).  The usual definition of $K_1(\boldA)$ is as the set of unitaries in matrix 
amplifications of $\boldA$ modulo homotopy equivalence, with the group operation being induced by the operation $\oplus$ introduced above.  It
would be interesting to see if perhaps one can show that this definition of $K_1$ leads to a computable functor which is computably naturally
equivalent to the functor just defined here.

\section{UHF algebras}\label{sec:UHF}

We now restrict our our attention to UHF algebras.  Recall that a unital \cstar-algebra $\boldA$ is called \emph{uniformly hyperfinite} (or \emph{UHF}) if there is a sequence of matrix algebras $(M_{n_k}(\mathbb{C}))_{k \in \mathbb{N}}$ and a sequence of unital embeddings $\iota_{k} : M_{n_k}(\mathbb{C}) \to M_{n_{k+1}}(\mathbb{C})$ so that $\boldA$ is isomorphic to $\displaystyle{\lim_{\rightarrow}(M_{n_k}(\mathbb{C}), \iota_k)}$.  We recall that there is a unital embedding from $M_m(\mathbb{C})$ to $M_n(\mathbb{C})$ if and only if $m$ divides $n$.  If $m$ divides $n$, then all such embeddings are unitarily conjugate to the canonical embedding $\mathcal{E}_{m,n}$ that takes a matrix $A$ in $M_m(\mathbb{C})$ to the block-diagonal matrix whose diagonal consists of $n/m$ copies of $A$.  More formally, 
\[
\embed_{m,n}(A) = \sum_{\ell < n/m} \sum_{r,s \in \{1, \ldots, m\}} a_{r,s} e^{(n)}_{r + \ell m, s + \ell m}.
\]  
As such, we frequently suppress mention of the embeddings and refer only to the sequence of matrix algebras involved in constructing a UHF algebra.  It is well-known that each UHF algebra $\boldA$ has a unique tracial state, which we denote by $\tau_{\boldA}$, given as the direct limit of the unique traces on the various matrix algebras.

It is well-known that UHF algebras are simple.  Consequently, by the result of Fox mentioned in Section \ref{sec:prelim}, any c.e. presentation of a UHF algebra is automatically computable.  In a sequel in preparation, we will show that this statement holds for the wider class of AF algebras (simple or otherwise).

\subsection{$K$-theory and supernatural numbers}
In \cite{Glimm.1960}, Glimm classified UHF algebras up to isomorphism by an invariant known as a \emph{supernatural number}.  This invariant is intimately related to the $K_0$ group of the algebra.  We recall the definitions and basic properties.  Throughout, $\primes$ denotes the set of prime numbers.

\begin{definition}
Let $\displaystyle{\boldA \cong \lim_{\rightarrow}M_{n_k}(\mathbb{C})}$ be a UHF algebra.  The \emph{supernatural number} of $\boldA$ is the function $\epsilon_{\boldA} : \primes \to \{0, 1, \ldots, \infty\}$ defined by 
\[\epsilon_{\boldA}(p) = \sup\{m \in \mathbb{N} : p^m \text{ divides } n_k \text{ for some $k$}\}.\]
\end{definition}

The supernatural number $\epsilon_{\boldA}$ is often expressed as a formal product $\prod_{p \in \primes}p^{n_p}$ where $n_p = \epsilon_{\boldA}(p)$.  However, for our purposes it is usually more convenient to think of it as a function, so that we may apply computability-theoretic terminology to it directly.

More generally, any function $\epsilon:\primes\to \N \cup \{\infty\}$ is referred to as a supernatural number.  It is easy to see that every supernatural number is the supernatural number of a UHF algebra (which is then necessarily unique up to isomorphism). 

Although the definition of $\epsilon_{\boldA}$ appears to depend on the choice of sequence of matrix algebras whose limit is $\boldA$, this is not in fact the case.

\begin{proposition}\label{prop:epsilon.embed}
Let $\boldA$ be a UHF algebra and fix a prime $p \in \primes$.  If $m \in \mathbb{N}$ is such that there is a unital embedding of $M_{p^m}(\mathbb{C})$ into $\boldA$, then $m \leq \epsilon_A(p)$.
\end{proposition}
\begin{proof}
Since $\boldA$ is UHF, we can express it as $\boldA \cong \lim_{\rightarrow}M_{n_k}(\mathbb{C})$.  Suppose that $\phi : M_{p^m}(\mathbb{C}) \to \boldA$ is a unital embedding and set $\sigma = \tau_{\boldA} \circ \phi$.  Then $\sigma$ is a tracial state on $M_{p^m}(\mathbb{C})$ and hence is the usual normalized trace.

Let $q \in M_{p^m}(\mathbb{C})$ be a rank $1$ projection.   Thus, $\sigma(q) = p^{-m}$.  Since $\phi(q)$ is a projection in $\boldA$, there exists $k$ and a projection $q' \in M_{n_k}$ such that $\phi(q)$ is unitarily conjugate (in $\boldA$) to $q'$.  In particular, $\tau_{\boldA}(q') = \tau_{\boldA}(\phi(q)) = \sigma(q) = p^{-m}$.  Thus the rank of $q'$ is $\tau_{\boldA}(q')n_k = p^{-m}n_k$.  Hence, $p^m$ divides $n_k$ and $m \leq \epsilon_{\boldA}(p)$.
\end{proof}

\begin{corollary}\label{cor:epsilon.sup}
For every UHF algebra $\boldA$ and every $p \in \primes$, we have
\[\epsilon_{\boldA}(p) = \sup\{m \in \mathbb{N} : \text{ there is a unital embedding of $M_{p^m}(\mathbb{C})$ into $\boldA$}\}.\]
\end{corollary}

The supernatural number can be viewed as an encoding of the ordered group $(K_0(\boldA), K_0(\boldA)^+, [\unit_{\boldA}]_{0})$, as we now recall.  For proofs of these claims, see \cite[Lemma 7.4.4.ii and Theorem 7.4.5.i]{Rordam.Larsen.Laustsen.2000}. 

Let $\boldA$ be a UHF algebra with supernatural number $\epsilon_{\boldA}$.  Let $\mathbb{Q}(\epsilon_{\boldA})$ be the subgroup of $(\mathbb{Q}, +)$ generated by numbers of the form $p^{-m}$, where $p$ is prime and $m \leq \epsilon_{\boldA}(p)$.   Then $K_0(\boldA) \cong \mathbb{Q}(\epsilon_{\boldA})$ via an isomorphism that sends the $K_0$-class of $\unit_{\boldA}$ to the rational number $1$ and sends $K_0(\boldA)^+$ to $\mathbb{Q}(\epsilon_{\boldA})^+$.  

On the other hand, every UHF algebra $\boldA$ has $K_1(\boldA) = 0$ (see \cite[Example 8.1.8 and Proposition 8.2.7]{Rordam.Larsen.Laustsen.2000}).  Thus supernatural numbers form a complete description of $K$-theory for UHF algebras.

\subsection{Computable UHF certificates}

Many different sequences of matrix algebras (and connecting maps) give rise to the same UHF algebra.  For computability purposes, it will be useful for us to have access not just to a UHF algebra, but also to data on how to construct the UHF algebra as a direct limit of matrix algebras. 
To this end, we introduce the following definition. 

\begin{definition}\label{def:UHF.cert}
A \emph{UHF certificate} for a \cstar-algebra $\boldA$ is a sequence $(n_j, \psi_j)_{j \in \N}$ that satisfies the following 
conditions:
\begin{enumerate}
    \item $(n_j)_{j \in \N}$ is a sequence of positive integers. 

    \item $\psi_j:M_{n_j}(\C)\to \boldA$ is a unital embedding.

    \item $\ran(\psi_j) \subseteq \ran(\psi_{j+1})$.

    \item $\boldA = \overline{\bigcup_{j \in \N} \ran(\psi_j)}$.
\end{enumerate}

The UHF certificate is a \emph{computable UHF certificate} of the presentation $\boldA^\#$ if the sequence $(n_j)_{j \in \mathbb{N}}$ is computable and the maps $\psi_j$ are computable maps of $M_{n_j}(\C)$ (with its standard presentation) into $\boldA^\#$, uniformly in $j$.
\end{definition}

It is clear that a \cstar-algebra $\boldA$ is UHF if and only if it has a UHF certificate.  Our next goal is to show that if $\boldA^\#$ is a computable presentation of a UHF algebra $\boldA$, then $\boldA^\#$ has a computable UHF certificate.  Note that the identity matrix of $M_n(\mathbb{C})$ is a rational point of the standard presentation of $M_n(\mathbb{C})$, and hence if $\boldA^\#$ has a computable UHF certificate, then $\unit_\boldA$ must be a computable point of $\boldA^\#$.  In fact, our construction of a computable UHF certificate for $\boldA^\#$ will use $\unit_\boldA$ as its starting point, so we first prove that $\unit_\boldA$ is a computable point of every computable presentation of a UHF algebra.  This result actually holds in greater generality.

\begin{proposition}\label{prop:UHF.unit}
If $\boldA$ is stably finite and $\boldA^\#$ is computable, then 
$\unit_A$ is a computable point of $\boldA^\#$.
\end{proposition}

\begin{proof}
Since $\mathcal{D}(\boldA^\#)$ is a presentation of $\mathcal{D}(\boldA)$, there is a 
$\mathcal{D}(\boldA^\#)$-label $w_0\in D_\omega$ for $[\unit_\boldA]_{\mvn}$. 
Since $\boldA^\#$ computably supports $\mathcal{D}(\boldA^\#)$, from $w_0$ it is possible to compute $n$ and an $M_n(\boldA^\#)$-index 
of a projection $p \in P_n(\boldA)$ so that $p \mvn \unit_\boldA$.

Next note that there is a computable projection $p'$ of $\boldA^\#$ so that $p' \mvn p$.  Indeed, by Theorem \ref{thm:mvn.ce}, there is a pair $(B_0, B_1)$ of rational open balls 
of $M_n(\boldA^\#)$ so that $\unit_\boldA \oplus \zerom_{n-1} \in B_0$, $p \in B_1$, and $q_0 \mvn q_1$ 
whenever $(q_0,q_1) \in (B_0 \times B_1) \cap P_n(\boldA)^2$.  Take a rational open ball 
$B_2$ of $\boldA^\#$ so that $\unit_\boldA \in B_2$ and 
$q\oplus \zerom_{n-1} \in B_0$ whenever $q \in B_2$.  Since $P_1(\boldA)$ is a c.e. closed set of 
$\boldA^\#$, and since $B_2 \cap P_1(\boldA) \neq \emptyset$, Lemma \ref{intersection} implies that $B_2 \cap P_1(\boldA)$ contains 
a computable point $p'$ of $\boldA^\#$.  Thus, $p' \mvn p' \oplus \zerom_{n-1} \mvn p$.

Since $\boldA$ is stably finite and $p'\mvn p\mvn 1_\boldA$, we have that $p'=1_\boldA$,as desired.
\end{proof}

Note that the proof of the previous proposition is \emph{not} uniform as it requires one to fix a label for the Murray-von Neumann equivalence class of the identity.  In fact, the uniform version of the statement is not true, as was shown by the third author in \cite[Theorem 5.3]{mcnicholl2024evaluative}.

We next show how to extend a matrix algebra in $\boldA$ to a larger one.  We begin by recalling a result of Glimm (see \cite[Lemma III.3.2]{Davidson.1996}), which we immediately strengthen in the lemma that follows it.

\begin{fact}\label{davidson}
For all $\epsilon>0$ and $n\in \mathbb{N}$, there is a positive real number 
$\delta=\delta(\epsilon,n)$ such that, whenever $\boldA$ and $\boldB$ are \cstar-subalgebras of a 
unital \cstar-algebra $\mathbf{D}$ with $\dim(A)\leq n$ and such that $\boldA$ has a system of matrix 
units $(e_{r,s})_{r,s}$ satisfying $d(e_{r,s},\boldB)<\delta$, then there is a unitary
$U\in C^*(\boldA,\boldB)$ with $\|U-1_{\mathbf{D}}\|<\epsilon$ for which $U\boldA U^*\subseteq \boldB$.
\end{fact}

\begin{lemma}\label{lm:UHF.ext}
Suppose that $\boldA$ is a UHF algebra, that $\bold M \subseteq \boldA$ is a finite dimensional 
subalgebra, and that $v_1, \ldots, v_k \in \boldA$.  Then for every $\epsilon > 0$, there is a 
\cstar-subalgebra $\boldB$ of $\boldA$ containing $\bold M$ such that $\boldB \cong 
M_n(\mathbb{C})$ for some $n$ and for which $d(v_i, \boldB) < \epsilon$ for all $i=1,\ldots,k$.
\end{lemma}

\begin{proof}  Let $\epsilon'$ be small enough so that for $(\epsilon'^2 + 2\epsilon')(\norm{v_i}) < \epsilon / 2$ for all $i=1,\ldots,k$.
Set $\delta = \delta(\epsilon', \dim(M))$ as in Fact \ref{davidson}, and let $\delta' = \min\{\delta, \epsilon', \epsilon/2\}$.  Let $(e_{r, s})_{r, s}$ be a system of matrix units for $\bold M$.  Since $\boldA$ is UHF, there is an algebra $\bold B' \subseteq \boldA$ such that $\bold B' \cong M_n(\C)$ for some $n$, 
$d(e_{r, s}, \bold B') < \delta'$ for all $r,s$, and $d(v_i, B') < \delta'$ for all $i=1,\ldots,k$.  By the choice of $\delta'$, there is a unitary $U$ in $C^*(\bold M, \bold B')$ such that $\norm{U-1_{\boldA}} < \epsilon'$ and $U\bold M U^* \subseteq \bold B'$.  Set $\boldB = U^*\bold B'U$.  Then 
$\bold M \subseteq \boldB$ and $\boldB \cong \boldB' \cong M_n(\mathbb{C})$.

For any $v\in \boldA$, we have
\begin{align*}
\norm{v-U^*vU} &= \norm{(U-1_{\boldA})^*v(U-1_{\boldA})+(U-1_{\boldA})^*v + v(U-1_{\boldA})} \\
&\leq (\norm{U-1_{\boldA}}^2+2\norm{U-1_{\boldA}})\norm{v} \\ &\leq (\epsilon'^2+2\epsilon')\norm{v}.
\end{align*}
In particular, for each $i=1,\ldots,k$, the choice of $\epsilon'$ gives $\norm{v_i-U^*v_iU} < \epsilon/2$.  Now for each $i=1,\ldots,k$, and each $b \in \bold B'$, we have
\begin{align*}
\norm{v_i-U^*bU} &= \norm{v_i-U^*(b-v_i)U - U^*v_iU} \\
 &\leq \norm{v_i - U^*v_iU} + \norm{b-v_i} \\
 &\leq \epsilon/2 + d(v_i, \bold B') \\
 &< \epsilon.
\end{align*}
Thus $d(v_i, \bold B) < \epsilon$ for all $i=1,\ldots,k$ as required.
\end{proof}

Our next step is a computable version of the previous lemma.

\begin{lemma}\label{lm:UHF.comp.ext}
Suppose $\boldA$ is UHF and $\boldA^\#$ is computable.  Assume $\phi$ is a unital embedding of 
$M_n(\C)$ into $\boldA$ and is a computable map from $M_n(\C)$ into $\boldA^\#$.
Then for every $k \in \N$ and every finite set $F$ of rational vectors of $\boldA^\#$, there exists
$n' > n$ with $n\mid n'$ and a unital embedding $\psi$ of $M_{n'}(\C)$ into $\boldA$ that satisfy the following conditions.
\begin{enumerate}
    \item $\psi$ is a computable map from $M_{n'}(\C)$ to $\boldA^\#$. \label{lm:UHF.comp.ext::comp}

    \item For all $\rho \in F$, $d(\rho, \ran(\psi)) < 2^{-k}$. \label{lm:UHF.comp.ext::dist}

    \item For all $r,s \in \{1, \ldots, n\}$,  \label{lm:UHF.comp.ext::units}
\[
\norm{ \phi(e_{r,s}^{(n)}) - \sum_{j < n'/n} \psi(e^{(n')}_{r + jn, s + jn})} < 2^{-k}.
\] 
\end{enumerate}
\end{lemma}

\begin{proof}
Set $\boldA_0 = \ran(\phi)$ and fix a positive integer $m$.  Let $C_m$ denote
the set of all $(mn) \times (mn)$ systems of matrix units in $M_{mn}(\boldA)$.  By Corollary \ref{cor:matrix.units.c.e.clsd.variant}, $C_m$ is a c.e. closed subset of $M_{mn}(\boldA)^\#$ uniformly in 
$m$.  Let $U^1_m$ denote the set of all 
$A \in M_{mn}(\boldA)$ so that
$d(\rho, \langle a_{r,s} \rangle_{r,s}) < 2^{-k}$ for each $\rho \in F$, where $\langle a_{r,s} \rangle_{r,s}$ denotes the subalgebra of $\boldA$ generated by the $a_{r,s}$'s.  
Let $U^2_m$ denote the set of all $A \in M_{mn}(\boldA)$ so that, 
for all $r,s \in \{1, \ldots, n\}$, one has
\[
\norm{\phi(e^{(n)}_{r,s}) - \sum_{j < m} a_{r + jn, s + jn}} < 2^{-k}.
\]

Note that both $U^1_m$ and $U^2_m$ are c.e. open sets of $M_{mn}(\boldA)^\#$ uniformly in $m$, whence so is 
$U^1_m \cap U^2_m$.

We claim that there is a positive integer $m$ so that $C_m \cap U^1_m \cap U^2_m \neq \emptyset$.  To see this, note that, by Lemma \ref{lm:UHF.ext}, there is an $n' \in \N$ and a \cstar-algebra 
$\boldB\cong M_{n'}(\C)$ so that $\boldA_0 \subseteq \boldB \subseteq \boldA$ and so that 
$d(\rho, \boldB) < 2^{-k}$ for all $\rho \in F$.  Let $\psi_0$ be a $*$-isomorphism 
of $M_{n'}(\C)$ onto $\boldB$.  Thus, $\psi_0^{-1}\circ \phi:M_n(\C)\to M_{n'}(\C)$ is a unital embedding.  Consequently, there is a unitary $U \in M_{n'}(\C)$ so that 
$\psi_0^{-1}(\phi(A)) = U\embed_{n,n'}(A)U^*$ for all $A \in M_n(\C)$. Now define $\psi_1:M_{n'}(\C)\to \boldA$ by $\psi_1(A)=\psi_0(UAU^*)$.  Set $a_{r,s} = \psi_1(e^{(n')}_{r,s})$, and let $A$ be the matrix with entries $a_{r,s}$.  It suffices to show $A \in C_m \cap U^1_m \cap U^2_m$.  Since $\psi_1$ is an embedding, we have $A \in C_m$.
Since we have $\langle a_{r,s} \rangle_{r,s} = \boldB$, it follows that $A \in U^1_m$.  Finally, since 
$\phi(e^{(n)}_{r,s}) = \psi_1(\embed_{n,n'}(e^{(n)}_{r,s})) = \sum_{j < m} a_{r + jn, s+ jn}$, we have $A \in U^2_m$, proving the claim.

Take $m$ as in the previous paragraph and set $n' = mn$.  By Lemma \ref{intersection}, there is a computable point $A$ belonging to  
$C_m \cap U^1_m \cap U^2_m$.  Let $\psi$ be the unital embedding of 
$M_{n'}(\C)$ into $\boldA$ for which $\psi(e^{(n')}_{r,s}) = a_{r,s}$.  
It follows that $\psi$ satisfies (\ref{lm:UHF.comp.ext::comp}) through 
(\ref{lm:UHF.comp.ext::units}).
\end{proof}

Once again, the proof of Lemma \ref{lm:UHF.comp.ext} shows that the conclusion holds uniformly in codes for $A^\#$ and $\phi$.

We now arrive at the key theorem of this subsection.

\begin{theorem}\label{thm:comp.UHF}
Suppose $\boldA$ is UHF and $\boldA^\#$ is computable.   
Then $\boldA^\#$ has a computable UHF certificate
$(n_j, \psi_j)_{j \in \N}$ so that $\psi_{j+1}^{-1}\psi_j = \embed_{n_j, n_{j+1}}$.  
\end{theorem}

\begin{proof}
We define a computable sequence $(n_j)_{j \in \N}$ of positive integers with $n_j\mid n_{j+1}$ and a monotonic nondecreasing sequence $(\boldA_j)_{j \in \N}$ for which 
$\boldA = \overline{ \bigcup_{j \in \N} \boldA_{n_j}}$.  At the same time, for each $j \in \N$, 
we will also define a $*$-isomorphism $\psi_j$ of $M_{n_j}(\C)$ onto $\boldA_j$ that is a computable map from $M_{n_j}(\C)$ to $\boldA^\#$ 
uniformly in $j$.  

These subalgebras and maps will be built by a sequence of approximations.  
More precisely, for each $j,t \in \N$ with $j \leq t$, we build an approximation 
$\boldA_{j,t}$ of $\boldA_j$.  This is acccomplished by defining an $n_j \times n_j$ system 
$G^{(j,t)} = (g^{(j,t)}_{r,s})_{r,s}$ of matrix units and setting 
$\boldA_{j,t} = \langle g^{(j,t)}_{r,s} \rangle_{r,s}$.  Along the way, we will ensure the following three conditions.
\begin{itemize}
    \item $1_\boldA = \sum_r g^{(j,t)}_{r,r}$.
    \item $G^{(j,t)}$ is a computable point of $M_{n_j}(\boldA)$ uniformly in $j,t$.
    \item $(G^{(j,t)})_{t \in \N}$ is a strongly Cauchy sequence of $M_{n_j}(\boldA)^\#$.  
    That is, $\norm{G^{(j,t)} - G^{(j,t+1)}} < 2^{-t}$ for all $t \in \N$.
\end{itemize}

 By the last item, we may set $G^{(j)} = \lim_t G^{(j,t)}$.  Writing $G^{(j)}=(g^{(j)}_{r,s})$,  we then define $\boldA_j = \langle g^{(j)}_{r,s} \rangle_{r,s}$. 
When $j \leq t$, we define $\psi_{j,t}:M_{n_j}(\C)\to \boldA_{j,t}$ by $\psi_{j,t}(e^{(n_j)}_{r,s}) = g^{(j,t)}_{r,s}$  
and then define $\psi_j:M_{n_j}(\C)\to \boldA_j$ by $\psi_j = \lim_t \psi_{j,t}$.   

An auxiliary sequence to be constructed during the construction is an array $(\alpha^{(n, t)}_{r,s})_{r,s}$ of 
 scalars from $\mathbb{Q}(i)$ whenever $n<t$.   
 
Let $(\rho_n)$ denote an effective enumeration of the rational points of $\boldA^\#$.  
There are two key invariants to be maintained at each stage of this construction.
{
\renewcommand{\theenumi}{INV-\arabic{enumi}}
\begin{enumerate}
    \item $\norm{ \rho_n - \sum_{r,s} \alpha^{(n, t')}_{r,s} g^{(t',t)}_{r,s} } < 2^{-t'}$ 
    when $n < t' \leq t$.  \label{inv1}

    \item $g^{(j,t)}_{r,s} = \sum_{\ell < n_{j+1}/n_j} g^{(j+1, t)}_{r + \ell n_j, s+ \ell n_j}$ when $j + 1 \leq t$.\label{inv2}
\end{enumerate}
}
At each stage of the construction, we assume the invariants are satisfied at all prior stages.\\  

\noindent\bf Stage 0:\rm\ We define $n_0$ to be $1$ and $g^{(0,0)}_{1,1} = \unit_\boldA$.\\

\noindent\bf Stage $\mathbf t + 1$:\rm\ We begin by setting $n=n_t$, $\phi=\psi_{t,t}$, $F=\{\rho_0,\ldots,\rho_t\}$, and 
$$k_{t+1}  = 
\log_2 \left( 2^{t+1} n_t \prod_{n < t' \leq t} 
\frac{n_t n_{t'} (\norm{\rho_n} + 2^{-t'} + 1)}{2^{-t'} - \norm{\rho_n - \sum_{r,s} \alpha^{(n,t')}_{r,s} g^{(t',t)}_{r,s} } }\right).$$

By Lemma \ref{lm:UHF.comp.ext}, there exists $n_{t+1} > n_t$ with $n_t\mid n_{t+1}$ and a unital embedding $\psi_{t+1,t+1}$ of 
$M_{n_{t+1}}(\C)$ into $\boldA$ that satisfy the following three conditions.  
\begin{enumerate}
    \item $\psi_{t+1, t+1}$ is a computable map of $M_{n_{t+1}}(\C)$ to $\boldA^\#$. 

    \item For all $n \leq t$, $d(\rho_n, \ran(\psi_{t+1,t+1})) < 2^{-k_{t+1}}$.

    \item For all $r,s \in \{1, \ldots, n_t\}$,
    \[
\norm{g_{r,s}^{(t,t)} - \sum_{j < n_{t+1}/n_t} \psi_{t+1,t+1}(e^{(n_{t+1})}_{r + jn_t, s + jn_t})} < 2^{-k_{t+1}}.
    \]
\end{enumerate}
We set $g_{r,s}^{(t+1,t+1)} = \psi_{t+1,t+1}(e^{(n_{t+1})}_{r, s })$.  
For each $n \leq t$, compute an array $(\alpha^{(n, t+ 1)}_{r,s})_{r,s}$ of
scalars from $\mathbb{Q}(i)$ so that 
\[
\norm{ \rho_n - \sum_{r,s} \alpha^{(n, t+ 1)}_{r,s} g_{r,s}^{(t+1,t+1)} } < 2^{-(t+1)}.
\]
For $j \leq t$, by ``backwards recursion'', we set 
\[
g^{(j,t+1)}_{r,s} = \sum_{\ell < n_{j+1}/n_j} g^{(j+1, t+1)}_{r + \ell n_j, s+ \ell n_j}.
\]

We first claim that for every $t \in \N$, 
\[
\norm{ g^{(j,t)}_{r,s} - g^{(j,t+1)}_{r,s} } \leq \frac{n_t}{n_j} 2^{-k_{t+1}}
\]
whenever $j \leq t$ and $r,s \in \{1, \ldots n_j\}$.  
For, let $t \in \N$.  We proceed by `backward induction' on $j$.  The base case 
$j = t$ holds from (3) above and the definition of $g_{r,s}^{(t,t+1)}$.  Now suppose $j < t$. By construction, 
\begin{eqnarray*}
g_{r,s}^{(j,t)} & = & \sum_{\ell < n_{j+1}/n_j } g^{(j+1,t)}_{r + \ell n_j, s + \ell n_j} \text{ and }\\
g_{r,s}^{(j,t+1)} & = & \sum_{\ell < n_{j+1}/n_j } g^{(j+1,t+1)}_{r + \ell n_j, s + \ell n_j} \\
\end{eqnarray*}
By the inductive hypothesis, 
\[
\norm{g^{(j,t)}_{r,s} - g^{(j, t+1)}_{r,s} }
\leq \frac{n_{j+1}}{n_j} \frac{n_t}{n_{j+1}} 2^{-k_{t+1}} 
= \frac{n_t}{n_j} 2^{-k_{t+1}}
\]
as required.

We now show that the invariants are satisfied at every stage $t$.
We first note that the invariants are vacuously true when $t = 0$.  
Suppose the invariants hold at stage $t$.  We now show they hold at stage $t+1$.
(\ref{inv2}) holds by construction. 
Suppose $n < t' \leq t+1$.  We need to show
\[
\norm{ \rho_n - \sum_{r,s} \alpha^{(n, t')}_{r,s} g^{(t',t+1)}_{r,s} } < 2^{-t'}.
\]
If $t' = t+1$, then this is by construction.  Suppose $t' \leq t$.  
On the one hand, 
\[
\norm{\rho_n - \sum_{r,s} \alpha^{(n, t')}_{r,s} g^{(t', t+1)}_{r,s} } \leq
\norm{\rho_n - \sum_{r,s} \alpha^{(n,t')}_{r,s} g^{(t',t)}_{r,s} } + 
\sum_{r,s} |\alpha_{r,s}| \norm{ g^{(t', t+1)}_{r,s} - g^{(t',t)}_{r,s} }.
\]
By what has just been shown, 
\[
\norm{g^{(t',t)}_{r,s} - g^{(t',t+1)}_{r,s} } \leq \frac{n_t}{n_{t'}} 2^{-k_{t+1}}.
\]
We also have that 
\[
|\alpha^{(n,t')}_{r,s}| \leq \norm{\sum_{r,s} \alpha^{(n,t')}_{r,s} e^{(n_{t'})}_{r,s} } 
= \norm{\sum_{r,s} \alpha^{(n,t')}_{r,s} g^{(t', t+ 1)}_{r,s} } \leq \norm{\rho_n} + 2^{-t'}.
\]
Thus, 
\[
\sum_{r,s} |\alpha^{(n,t')}_{r,s}| \norm{ g^{(t',t+1)}_{r,s} - g^{(t',t)}_{r,s} } 
\leq (n_{t'})^2 (\norm{\rho_n} + 2^{-t'} ) \frac{n_t}{n_{t'}} 2^{-k_{t+1}} 
= n_t n_{t'} (\norm{\rho_n} + 2^{-t'} )  2^{-k_{t+1}}.
\]
By the definition of $k_{t+1}$, 
\[
n_t n_{t'} (\norm{\rho_n} + 2^{-t'} ) 2^{-k_{t+1}} < 2^{-t'} - \norm{ \rho_n - \sum_{r,s} \alpha^{(n, t')}_{r,s} g^{(t',t)}_{r,s} }.
\]
Hence, the invariants are maintained at all stages of the construction.

The only thing that remains to be verified is that $\boldA = \overline{\bigcup_{n \in \N} \boldA_n}$.  To see this, note that
$d(\rho_n, \boldA_{t}) \leq 2^{-t}$ whenever $n < t$.  
Thus, every rational vector of $\boldA^\#$ belongs to $\overline{\bigcup_{n \in \N} \boldA_n}$, which clearly suffices.
\end{proof}

\begin{remark}
If $\boldA^\#$ is a computable presentation of the UHF algebra $\boldA$ and $(n_j,\psi_j)$ is a computable UHF certificate of $\boldA^\#$, then it is clear that one may compute, for every rational $\epsilon>0$ and rational point $y$ of $A^\#$, an integer $j$ and a rational point $x$ of $M_{n_j}(\mathbb{C})$ so that $\|\psi_j(x)-y\|<\epsilon$.  In other words, the union of the ranges of the $\psi_j$'s is ``computably dense.''  Conversely, if $A^\#$ is a presentation of the UHF algebra $\boldA$ for which there is a computable UHF certificate of $\boldA^\#$ satisfying this extra computable density property, then $\boldA^\#$ is readily seen to be computable.
\end{remark}

\subsection{Computability of $K$-theory}

We saw in Corollary \ref{cor:K0.comp.func} that if $\boldA^\#$ is a computable presentation of a unital \cstar-algebra $\boldA$, then $K_0(\boldA^\#)$ is a c.e. presentation of $K_0(\boldA)$. Moreover, Proposition \ref{prop:cone.ce} implies that if $\boldA$ is UHF, then the positive cone $K_0(\boldA)^+$ is a c.e. subset of $K_0(\boldA^\#)$ (since UHF algebras are stably finite).  We now show that for UHF algebras, the situation is better, in that $K_0(\boldA^\#)$ is actually a computable presentation of $K_0(\boldA)$ and the positive cone $K_0(A)^+$ is a computable subset of $K_0(\boldA^\#)$.

\begin{theorem}\label{thm:comp.trace}  
Suppose $\boldA$ is a UHF algebra.  If $\boldA^\#$ is computable, then 
the tracial state $\tau_\boldA$ of $\boldA$ is a computable map from $\boldA^\#$ to $\C$.
\end{theorem}

\begin{proof}
By Theorem \ref{thm:comp.UHF}, there is a computable UHF certificate $(n_j,\psi_j)_{j\in \N}$ for $\boldA^\#$.
For each $j \in \mathbb{N}$, let $\boldA_j = \ran(\psi_j)$ and let $\tau_j$ be the tracial state of $M_{n_j}(\C)$.  Note that the maps $\tau_j$ are computable, uniformly in $j$.

Given a rational open ball $B(\rho; r)$ of $\boldA^\#$, we compute a rational open disk 
$D(\zeta; 3r) \subseteq \C$ as follows.  First, we compute $j$ and a rational point
$A \in M_{n_j}(\C)$ so that $\norm{\psi_j(A) - \rho} < r$, whence it follows that 
\[|\tau_j(A) - \tau(\rho)| = |\tau\psi_j(A) - \tau(\rho)| < r.\]  
Next, since $\tau_j$ is computable, we can compute $\zeta \in \Q(i)$ so that $|\zeta - \tau_j(A)| < r$.  
As a result, for each $a \in B(\rho; r)$, we have
\[
|\tau(a) - \zeta| \leq |\tau(a) - \tau(\rho)| + |\tau(\rho) - \tau_j(A)| + |\zeta - \tau_j(A)| < 3r.
\]
Thus, $\tau$ is a computable map of $\boldA^\#$ into $\C$. 
\end{proof}

\begin{definition}
Suppose $\tau$ is a trace on a \cstar-algebra $\boldA$.  We say that $(\boldA,\tau)$ is \emph{factor-like} provided that for any two projections $p,q\in P_{<\omega}(\boldA)$, $p\mvn q$ if and only if $\tau(p)=\tau(q)$.
\end{definition}

The nomenclature stems from the fact that the conclusion of the definition holds for tracial factors.  Note that $\tau(p)=\tau(q)$ always holds when $p\mvn q$, whence the import of the definition is the statement that projections of the same trace are Murray-von Neumann equivalent.

The following is probably well-known, but since we could not find it explicitly stated in the literature, we include a proof.

\begin{lemma}
Suppose that $\boldA$ is UHF and $\tau_{\boldA}$ is its unique tracial state.  Then $(\boldA,\tau_{\boldA})$ is factor-like.    
\end{lemma}

\begin{proof}
Since $M_n(\boldA)$ is UHF as well, it suffices to consider the case that $p,q\in P(\boldA)$ are such that $\tau_{\boldA}(p)=\tau_{\boldA}(q)$.  Take a \cstar-subalgebra $\boldB\subseteq \boldA$ such that $\boldB\cong M_n(\C)$ for some $n$ and which contains projections $p'$ and $q'$ such that $\|p-p'\|,\|q-q'\|<1/2$.  Thus, $p\mvn p'$, $q\mvn q'$, and consequently 
$\tau_{\boldA}(p')=\tau_{\boldA}(q')$.  However, $\tau_{\boldA}$ agrees with the usual trace on $M_n(\C)$ (once $M_n(\C)$ is identified with $\boldB$), and it is well-known that matrix algebras equipped with their usual trace are factor-like (in fact, they are factors!).
\end{proof}

The following is immediate from the definitions.

\begin{proposition}\label{factorlike}
Suppose that $(\boldA,\tau)$ is factor-like  and that $\boldA^\#$ is a presentation such that $\tau$ is a computable map from $A^\#$ to $\C$.  Then there is a c.e. set $R\subseteq \N^2$ such that, for all  $e_0, e_1 \in \N$, 
if $e_j$ is a $\boldA^\#$-index of a projection $p_j\in P_{<\omega}(\boldA)$, then 
$p_0 \not\mvn p_1$ if and only if $R(e_0,e_1)$.
\end{proposition}




\begin{theorem}\label{thm:K0.UHF}
If $\boldA^\#$ is a computable presentation of a UHF algebra $\boldA$, then
$K_0(\boldA^\#)$ is a computable presentation of $K_0(\boldA)$ and $K_0(\boldA)^+$ is a computable subset of 
$K_0(\boldA^\#)$.
\end{theorem}

\begin{proof}
By Proposition \ref{Grothendieckcomputable}, in order to show that $K_0(\boldA^\#)$ is computable, it suffices to show that $\mathcal{D}(\boldA^\#)$ is computable.  Towards this end, fix $w_0, w_1 \in D_\omega$.  
Since $\mathcal{D}(\boldA^\#)$ is computably supported by $\boldA^\#$, for $j=0,1$
we can compute $n_j \in \N$ and an $\boldA^\#$-index of a projection $p_j \in M_{n_j}(\boldA)$
for which $w_j$ is a $\mathcal{D}(\boldA^\#)$-label.  Let $n = \max\{n_0, n_1\}$.  Since 
$(\boldA,\tau_{\boldA})$ is factor-like, Theorem \ref{thm:mvn.ce} and Proposition \ref{factorlike} 
allow us to determine if $p_0 \mvn p_1$, as required to show that $\mathcal{D}(\boldA^\#)$ is 
computable.

Finally, since $\boldA$ is UHF, the ordering on $K_0(\boldA)$ is linear,
and hence $K_0(\boldA)^+$ is a computable subset of $K_0(\boldA^\#)$ by Corollary \ref{linearcomputable}.
\end{proof}


We now turn to the connection with supernatural numbers.  The following definition is 
an adaptation of a standard concept from computable analysis.

\begin{definition}
\label{defn:frombelow}
A supernatural number $\epsilon$ is \emph{lower semicomputable} if 
it is the pointwise limit of a uniformly computable 
sequence $(h_j)_{j \in \N}$ of functions from $\mathbb{P}$ into $\N$ 
so that $h_j(p) \leq h_{j + 1}(p)$ for all $j \in \N$ and $p \in \mathbb{P}$.
\end{definition}

We are now prepared to demonstrate the relations between computability properties of $\boldA$, $K_0(\boldA)$, and $\epsilon_{\boldA}$.
\begin{theorem}\label{thm:UHF.comp.pres}
Suppose $\boldA$ is a UHF algebra.  The following are equivalent.
\begin{enumerate}
    \item $\boldA$ has a computable presentation. \label{thm:UHF.comp.pres::comp.pres}
    \item $K_0(\boldA)$ has a computable presentation. \label{thm:UHF.comp.pres::K0.comp.pres}
    \item $K_0(\boldA)$ has a c.e. presentation.\label{thm:UHF.comp.pres::K0.ce.pres}
    \item The supernatural number of $\boldA$ is lower semicomputable.\label{thm:UHF.comp.pres::super}
\end{enumerate}
\end{theorem}

\begin{proof}
Theorem \ref{thm:K0.UHF} already states that 
(\ref{thm:UHF.comp.pres::comp.pres}) implies (\ref{thm:UHF.comp.pres::K0.comp.pres}).
It is definitional that (\ref{thm:UHF.comp.pres::K0.comp.pres}) implies (\ref{thm:UHF.comp.pres::K0.ce.pres}). 
For (\ref{thm:UHF.comp.pres::K0.ce.pres}) $\Rightarrow$ (\ref{thm:UHF.comp.pres::super}), set $G = K_0(\boldA)$, and let $G^\#$ be a c.e. presentation of $G$.
Since $\boldA$ is UHF, we may assume $G = \Q(\epsilon_A)$ in that, from the presentation $G^\#$, we may form a c.e. presentation of 
$\Q(\epsilon_\boldA)$ that has the same index.   
Fix a $G^\#$-label $w_0$ of $1$ and set 
\[
E = \{(p,m)\in \primes \times \N \ :\ \mbox{$M_{p^m}(\C)$ unitally $\star$-embeds into $\boldA$}\}.
\]
By Corollary \ref{cor:epsilon.sup}, to show that $\epsilon_\boldA$ is lower semicomputable, it suffices to show $E$ is c.e.  
Let $\kappa$ denote the kernel of $G^\#$.  Fix a prime $p$ and an $m \in \N$.  
By the definition of $\Q(\epsilon_\boldA)$, 
$m \leq \epsilon_\boldA(p)$ if and only if 
$(w^{p^m}, w_0) \in \kappa$ for some $w \in F_\omega$.  
At the same time, by Corollary \ref{cor:epsilon.sup}, 
$m \leq \epsilon_\boldA(p)$ if and only if $(p,m) \in E$.  
Since $\kappa$ is c.e., it follows that $E$ is c.e.

Finally, suppose that $\epsilon_\boldA$ is lower semicomputable and take a nondecreasing and uniformly computable sequence $(h_j)_{j \in \N}$ of functions from $\primes$ to $\N$
so that $\epsilon_\boldA = \lim_j h_j$.  Set $g_j = h_j \cdot \chi_{[0, j]}$, where $\chi_{[0,j]}$ is the characteristic function of $[0,j]$.  Then, set 
\[
n_j = \prod_{p \in \primes} p^{h_j(p)}.
\]
Let $\mathbf{B}$ be the inductive limit of $(M_{n_j}(\C),\embed_{n_j,n_{j+1}})$, which is is isomorphic to $\boldA$ by construction.  
By \cite[Lemma 2.7.3]{Goldbring.2024+}, $\mathbf{B}$ is computably presentable, whence so is $\boldA$.
\end{proof}

It follows from the previous theorem that $\mathbf{0''}$ computes the supernatural number of every 
computably presentable \cstar algebra.  
We now give an example of a supernatural number that achieves this upper bound.
Let $(p_e)_{e \in \N}$ be the increasing enumeration of $\primes$.  
For each $e,s$, let $h_s(p_e) = \# W_{e,s}$ and set $\epsilon(p) = \lim_s h_s(p)$.
Thus, by definition, $\epsilon$ is a lower-semicomputable supernatural number.
However, $W_e$ is infinite if and only if $\epsilon(p_e) = \infty$.
At the same time, the Turing degree of 
$\{e\ :\ W_e\ \mbox{is infinite}\}$ is $\mathbf{0''}$.  Therefore, 
any oracle that computes $\epsilon$ computes $\mathbf{0''}$ as well.

\subsection{Computable categoricity}

Recall that a structure is called \emph{computably categorical} if every two computable presentations of the structure are computably isomorphic.  We conclude this paper by showing that all UHF algebras enjoy this property.

\begin{theorem}\label{thm:UHF.comp.cat}
Every UHF algebra is computably categorical.
\end{theorem}

\begin{proof}
Suppose $\boldA$ is UHF and suppose that $\boldA^\#$ and $\boldA^\dagger$ are computable presentations of $\boldA$.  By Theorem \ref{thm:comp.UHF}, $\boldA^\#$ has a computable UHF certificate
$(m_j, \phi_j)_{j \in \N}$ satisfying $\phi_{j+1}^{-1}\phi_j = \embed_{m_j, m_{j+1}}$.  
Similarly, $\boldA^\dagger$ has a computable UHF certificate
$(n_j, \psi_j)_{j \in \N}$ satisfying $\psi_{j+1}^{-1}\psi_j = \embed_{n_j, n_{j+1}}$. 
Set $\boldA_j = \ran(\phi_j)$ and $\boldA'_j = \ran(\psi_j)$.

We define sequences $(k_j)_{j \in \N}$ and $(\ell_j)_{j \in \N}$ of positive integers 
by simultaneous recursion as follows.  Set $k_0 = 0$ and let $\ell_j$ be the smallest number greater than $\max\{\ell_{j'}\ :\ j' < j\}$
for which $m_{k_j} | n_{\ell_j}$.
Similarly, let $k_{j+1}$ be the smallest number greater than $\max\{\ell_{k'}\ :\ j' \leq j\}$ for which 
$n_{\ell_j} | m_{k_{j+1}}$.  

Now define $\gamma_j:\boldA_{k_j}\to \boldA'_{l_j}$ and $\delta_j:\boldA'_{l_j}\to \boldA_{k_{j+1}}$ by setting

\begin{eqnarray*}
\gamma_j & = & \psi_{\ell_j} \embed_{m_{k_j}, n_{\ell_j}} \phi_{k_j}^{-1}\mbox{, and}\\
\delta_j & = & \phi_{k_{j+1}} \embed_{n_{\ell_j}, m_{k_{j+1}}} \psi^{-1}_{\ell_j}.
\end{eqnarray*}

Note that $\delta_j\gamma_j : \boldA_{k_j} \to \boldA_{k_{j+1}}$ is the inclusion map $\iota_{k_j, k_{j+1}}$ and likewise that $\gamma_{j+1}\delta_j : \boldA'_{\ell_j} \to \boldA'_{\ell_{j+1}}$ is the inclusion map $\iota'_{\ell_j, \ell_{j+1}}$.  In particular, $\gamma_{j+1}$ extends $\gamma_j$ and $\delta_{j+1}$ extends $\delta_j$.

Let $\gamma = \bigcup_{j \in \N} \gamma_j$, and let $\delta = \bigcup_{j \in \N} \delta_j$.
Since $\gamma_j$ and $\delta_j$ are embeddings, they are $1$-Lipschitz.  
It follows that $\gamma$ has a unique continuous extension $\overline{\gamma}$ to 
$\boldA$ and that $\delta$ has a unique continuous extension $\overline{\delta}$ to 
$\boldA$.  Since each $\gamma_j$ and $\delta_j$ is an embedding, it follows that the extensions $\overline{\gamma}$ and $\overline{\delta}$ are $*$-homomorphisms.  As we have already seen that $\delta_j\gamma_j = \iota_{k_j, k_{j+1}}$ and $\gamma_{j+1}\delta_j = \iota'_{\ell_j, \ell_{j+1}}$, it follows that $\overline{\gamma} = \overline{\delta}^{-1}$.  In particular, $\overline{\gamma}$ and $\overline{\delta}$ are each automorphisms of $\boldA$.

It remains to show that $\overline{\gamma}$ is a computable map from $\boldA^\#$ to $\boldA^\dagger$ and $\overline{\delta}$ is a computable map from $\boldA^\dagger$ to $\boldA^\#$.  As the two proofs are similar, we will only establish the former claim.

Let $\rho$ be a rational point of $\boldA^\#$ and fix $k \in \N$. 
Search for $j \in \N$ and a rational point $\rho'$ of 
$M_{m_{k_j}}(\C)$ so that $\norm{\phi_{k_j}(\rho') - \rho} < 2^{-(k+1)}$.  Since $\boldA = \overline{\bigcup_{j \in \N} \boldA_j}$ 
and $\lim_j k_j = \infty$, 
this search must terminate.  
Since $\overline{\gamma}$ is $1$-Lipschitz, we have
$$\norm{\overline{\gamma}(\rho) - \overline{\gamma}(\phi_{k_j}(\rho'))} < 2^{-(k+1)}.$$ 
Moreover, we have
\[\overline{\gamma}(\phi_{k_j}(\rho')) = \gamma_j(\phi_{k_j}(\rho')) = \psi_{\ell_j}\embed_{k_j, \ell_j}(\rho').\]  
The point $\psi_{\ell_j}\embed_{k_j, \ell_j}(\rho')$ is a computable point of $\boldA^\dagger$ uniformly in $j$ and $\rho'$.
Thus, we may compute a rational point $\rho''$ of $\boldA^\dagger$ so that 
$\norm{\rho'' - \overline{\gamma}(\phi_{k_j}(\rho'))} < 2^{-(k+1)}$.  Thus, 
$\norm{\overline{\gamma}(\rho) - \rho''} < 2^{-k}$ as desired.
\end{proof}


\section*{Acknowledgements}  
The work presented here was initiated during the 5-day workshop 23w505 at the Banff International Research Station (BIRS).  The authors thank BIRS for providing an excellent opportunity for collaboration.  
The first-named author also wishes to acknowledge the hospitality of Iowa State University during a visit to the third-named author.

\appendix
\section{The functor $K_0$ in general}
    
It is naturally desirable to extend $K_0$ to a ``computable'' functor from the category of presentations of unital $C^*$-algebras (computable or otherwise)
to the category of semigroup presentations.  However, for any reasonable definition of what it means for a functor of this sort to be computable, 
it seems unlikely that such a thing exists.  Indeed, if it did, it would map computable presentations of unital \cstar-algebras to computable presentations of abelian groups and we believe it is not the case that computable presentability of $\boldA$ implies the computable presentability of $K_0(\boldA)$ (although we have no counterexample at
present).  Consequently, we propose a weaker condition than computability of the functor for which our methods do naturally extend.

We consider three categories of presentations: 
\begin{enumerate}
    \item The category of presentations of semigroups.  In this category, the objects are the 
    semigroup presentations and the morphisms are the computable homomorphisms between semigroup presentations. 

    \item The category of presentations of groups.  In this category, the objects are the group presentations 
    and the morphisms are the computable maps between group presentations. 

    \item The category of presentations of \cstar algebras.  In this category, the objects are the 
  presentations of \cstar algebras and the morphisms are the computable $*$-homomorphisms between presentations.  
\end{enumerate}

Our definition of computability for functors between these categories is based on 
names and name transformations which we define now.

\begin{definition}\label{def:name}
A \emph{name} is a pair $(e,Y) \in \N \times \mathcal{P}(\N)$.
\end{definition}

\begin{definition}\label{def:name.trans}
A \emph{name transformation} is a map $F: \N \times \mathcal{P}(\N) \rightarrow \N \times \mathcal{P}(\N)$ 
so that, whenever $F(e, X_0) = (e'_0, Y_0)$ and $F(e, X_1) = (e_1', Y_1)$, then $e_0' = e_1'$.  
\end{definition}

In the case of semigroup and group presentations, we consider three kinds of names.

\begin{definition}\label{def:names.semig}
Suppose $J^\#$ is a semigroup presentation or a group presentation and let $(e,X)$ be a name.
\begin{enumerate}
    \item $(e,X)$ is a \emph{purely positive name} for $J^\#$ if 
    the kernel of $J^\#$ is $W_e^X$.  

    \item $(e,X)$ is a \emph{purely negative name} for $J^\#$ if 
    the kernel of $J^\#$ is the complement of $W_e^X$.

    \item $(e,X)$ is an \emph{exact name} for $J^\#$ if 
    the kernel of $J^\#$ is computed by $\phi_e^X$.
\end{enumerate}
\end{definition}

We analogously define three kinds of names for \cstar-algebras.

\begin{definition}\label{def:names.alg}
Let $(e,X)$ be a name and
let $\boldA^\#$ be a presentation of a \cstar algebra. 
\begin{enumerate}
    \item $(e,X)$ is a \emph{purely positive name} for $\boldA^\#$ if, 
    for every rational point $\rho$ of $\boldA^\#$, $W_{\phi_e(\rho)}^X$
    is the left Dedekind cut of $\norm{\rho}$.  

    \item $(e,X)$ is a \emph{purely negative name} for $\boldA^\#$ if, 
    for every rational point $\rho$ of $\boldA^\#$, $W_{\phi_e(\rho)}^X$ 
    is the right Dedekind cut of $\norm{\rho}$.

    \item $(e,X)$ is an \emph{exact name} for $\boldA^\#$ if 
    $\phi_e^X$ computes the norm of $\boldA$ on the rational points of 
    $\boldA^\#$.
\end{enumerate}
\end{definition}

Rather than a single concept of a computable functor, we now propose three concepts based on the kinds of 
names used for the objects in the domain and range.

\begin{definition}\label{def:comp.funct}
Suppose that, for each $j \in \{0,1\}$, $\mathcal{C}_j$ is one of the three categories listed above.
Suppose $F$ is a functor from $\mathcal{C}_0$ to $\mathcal{C}_1$.
\begin{enumerate}
    \item $F$ is \emph{lower semicomputable} if there is a computable name transformation $\Phi$ so that 
    $\Phi(\eta)$ is a purely positive name of $F(J^\#)$ whenever $\eta$ is an exact name of $J^\#$. 

    \item $F$ is \emph{upper semicomputable} if there is a computable name transformation $\Phi$ so that 
    $\Phi(\eta)$ is a purely negative name of $F(J^\#)$ whenever $\eta$ is an exact name of $J^\#$.

    \item $F$ is \emph{strongly lower semicomputable} if there is a computable name transformation $\Phi$ so that 
    $\Phi(\eta)$ is a purely positive name of $F(J^\#)$ whenever $\eta$ is a purely positive name of $J^\#$. 
\end{enumerate}
In each of the above cases, we also require that, from an index of a morphism $\phi$ from 
$J_0^\#$ to $J_1^\#$, it is possible to compute an index of $F(\phi)$.
\end{definition}

In computable analysis, a sequence $(a_n)_{n \in \N}$ of reals is lower semicomputable if, from 
$n\in \N$, it is possible to compute an enumeration of the left Dedekind cut of $a_n$.  
Upper semicomputability of a sequence of reals is defined similarly.  These concepts motivate
the choice of terminology in the above definition.

Finally, the uniformity of the arguments in our paper lead to the following:

\begin{theorem}\label{thm:ultra.unif}

\

\begin{enumerate}
    \item $\mathcal{D}$ extends to a strongly lower semicomputable functor from the category of \cstar-algebra presentations
     to the category of semigroup presentations. 

     \item $\mathcal{G}$ extends to a strongly lower semicomputable functor from the category of abelian semigroup
     presentations to the category of abelian group presentations. 

     \item $K$ extends to a strongly lower semicomputable functor from the category of \cstar-algebra presentations to the 
     category of group presentations.
\end{enumerate}
Furthermore, there is a computable function $f : \N \rightarrow \N$ so that, whenever $(e,X)$ is an exact name of a 
presentation $\boldA^\#$ of a stably finite algebra $\boldA$, we have that
$W^X_{f(e)} = \{w\ :\ \mbox{$w$ is a $K^c_0(\boldA^\#)$-notation for an element of $K_0(\boldA)^+$}\}$.
\end{theorem}

\bibliographystyle{amsplain}
\bibliography{paperbib}

\end{document}